\documentclass[a4paper,11pt]{article}
\usepackage{amsmath,amsthm,amssymb,enumitem,xcolor}

\usepackage[hang]{footmisc}
\usepackage[nosort,nocompress,noadjust]{cite}

\usepackage[bookmarks=false,hyperfootnotes=false,colorlinks,
    linkcolor={red!60!black},
    citecolor={blue!50!black},
    urlcolor={blue!80!black}]{hyperref}

\renewcommand{\eqref}[1]{\hyperref[#1]{(\ref{#1})}}

\newlist{enumlist}{enumerate}{2}
\setlist[enumlist,1]{labelindent=0cm,label=\arabic*.,ref=\arabic*,labelwidth=2.5ex,labelsep=0.5ex,leftmargin=3ex,align=left,topsep=0.5ex,itemsep=1ex,parsep=1ex}
\setlist[enumlist,2]{labelindent=0cm,label=\theenumlisti.\arabic*.,ref=\arabic*,labelwidth=5ex,labelsep=0.5ex,leftmargin=5.5ex,align=left,topsep=0.5ex,itemsep=1ex,parsep=1ex}

\newlist{itemlist}{itemize}{1}
\setlist[itemlist]{labelindent=0cm,label=$\bullet$,labelwidth=2.5ex,labelsep=0.5ex,leftmargin=3ex,align=left,topsep=0.5ex,itemsep=1ex,parsep=1ex}

\numberwithin{equation}{section}

{\theoremstyle{definition}\newtheorem{definition}{Definition}[section]

\newtheorem{example}[definition]{Example}}

\newtheorem{proposition}[definition]{Proposition}
\newtheorem{lemma}[definition]{Lemma}
\newtheorem{theorem}[definition]{Theorem}
\newtheorem{corollary}[definition]{Corollary}
\newtheorem{letterthm}{Theorem}

\newtheorem{letterprop}[letterthm]{Proposition}

\newcommand{\core}{\operatorname{c}}
\newcommand{\om}{\omega}
\newcommand{\si}{\sigma}
\newcommand{\actson}{\curvearrowright}
\newcommand{\R}{\mathbb{R}}
\newcommand{\N}{\mathbb{N}}
\newcommand{\Stab}{\operatorname{Stab}}
\newcommand{\cZ}{\mathcal{Z}}
\newcommand{\cF}{\mathcal{F}}
\newcommand{\cP}{\mathcal{P}}
\newcommand{\eps}{\varepsilon}
\newcommand{\class}{\operatorname{class}}
\newcommand{\Ad}{\operatorname{Ad}}
\newcommand{\cU}{\mathcal{U}}
\newcommand{\ot}{\otimes}
\newcommand{\vphi}{\varphi}
\newcommand{\mutil}{\widetilde{\mu}}
\newcommand{\gammatil}{\widetilde{\gamma}}
\newcommand{\cR}{\mathcal{R}}
\newcommand{\cG}{\mathcal{G}}
\newcommand{\Tr}{\operatorname{Tr}}
\newcommand{\Z}{\mathbb{Z}}
\newcommand{\stc}{\operatorname{stc}}
\newcommand{\dpr}{^{\prime\prime}}

\newcommand{\al}{\alpha}
\newcommand{\be}{\beta}

\newcommand{\sigmatil}{\widetilde{\sigma}}
\newcommand{\F}{\mathbb{F}}
\newcommand{\cS}{\mathcal{S}}
\newcommand{\C}{\mathbb{C}}
\newcommand{\cN}{\mathcal{N}}
\newcommand{\Norm}{\operatorname{Norm}}
\newcommand{\ovt}{\mathbin{\overline{\otimes}}}
\newcommand{\sdif}{\vartriangle}
\newcommand{\Aut}{\operatorname{Aut}}

\begin{document}

\begin{center}
{\boldmath\LARGE\bf Classification results for nonsingular\vspace{0.5ex}\\ Bernoulli crossed products}

\bigskip

{\sc by Stefaan Vaes\footnote{\noindent KU~Leuven, Department of Mathematics, Leuven (Belgium).\\ E-mails: stefaan.vaes@kuleuven.be and bram.verjans@kuleuven.be.}\textsuperscript{,}\footnote{S.V.\ is supported by FWO research project G090420N of the Research Foundation Flanders and by long term structural funding~-- Methusalem grant of the Flemish Government.} and Bram Verjans\textsuperscript{1,}\footnote{B.V.\ is holder of PhD grant 1139721N fundamental research of the Research Foundation~-- Flanders.}}

\end{center}

\begin{abstract}
\noindent We prove rigidity and classification results for type III factors given by nonsingular Bernoulli actions of the free groups and more general free product groups. This includes a large family of nonisomorphic Bernoulli crossed products of type III$_1$ that cannot be distinguished by Connes $\tau$-invariant. These are the first such classification results beyond the well studied probability measure preserving case.
\end{abstract}

\section{Introduction}

In the past years, Popa's deformation/rigidity theory has led to a broad range of rigidity theorems for probability measure preserving (pmp) Bernoulli actions $\Gamma \actson (X,\mu) = (X_0,\mu_0)^\Gamma$, see e.g.\ \cite{Pop03,Pop04,Pop06,Ioa10,IPV10,PV21}. This includes numerous W$^*$-superrigidity results showing that both the group $\Gamma$ and its action $\Gamma \actson (X,\mu)$ can be entirely retrieved from the ambient II$_1$ factor $L^\infty(X) \rtimes \Gamma$.

More recently, there has been a growing interest in \emph{nonsingular} Bernoulli actions $\Gamma \actson (X,\mu) = \prod_{g \in \Gamma} (X_0,\mu_g)$, where the base measures $\mu_g$ vary. For $\Gamma = \Z$, this provides under the appropriate assumptions a classical family of nonsingular ergodic transformations that have been widely studied, see e.g.\ \cite{Ham81,Kos09,BKV19}. For nonamenable groups $\Gamma$, a first systematic study of nonsingular Bernoulli actions was made in \cite{VW17}. In view of the wealth of rigidity theorems for pmp Bernoulli actions, this raises the natural problem to prove rigidity and classification theorems for the type III factors $L^\infty(X,\mu) \rtimes \Gamma$ associated with nonsingular Bernoulli actions.

There is a conceptual reason why obtaining such rigidity theorems is a hard problem. It was proven in \cite[Theorem 3.1]{VW17} that if a nonamenable group $\Gamma$ admits a nonsingular Bernoulli action of type III, then $\Gamma$ must have a nonzero first $L^2$-Betti number. In the pmp setting, all superrigidity theorems for Bernoulli actions are restricted to nonamenable groups with zero first $L^2$-Betti number! It has even been conjectured that a pmp Bernoulli action satisfies cocycle superrigidity (w.r.t.\ the appropriate target groups) if and only if $\Gamma$ is nonamenable with zero first $L^2$-Betti number.

Therefore, we only set out to prove strong rigidity theorems, providing partial classification results for natural families of type III Bernoulli crossed products $L^\infty(X,\mu) \rtimes \Gamma$. These are the first classification results for type III Bernoulli crossed products going beyond Connes $\tau$-invariant to distinguish between such type III factors. We specifically prove these results for the wide family of nonsingular Bernoulli actions of the free groups $\F_n$ that were introduced in \cite[Section 7]{VW17} and that we recall below. More generally, we consider such Bernoulli actions for arbitrary free product groups $\Gamma = \Z \ast \Lambda$ with $\Lambda$ being any nonamenable group.

Recall from \cite[Section 7]{VW17} that, given any standard Borel space $Y$ and equivalent probability measures $\nu \sim \eta$ on $Y$, and given any countably infinite group $\Lambda$, we can consider the nonsingular Bernoulli action of the free product $\Gamma = \Z \ast \Lambda$ given by
\begin{equation}\label{eq.my-bernoulli}
\begin{split}
\Gamma \actson (X,\mu) = \prod_{h \in \Gamma} (Y,\mu_h) &\;\;\text{where}\;\; \mu_h = \begin{cases} \nu &\;\text{if the last letter of $h$ belongs to $\N \subset \Z$,}\\
\eta &\;\text{otherwise,}\end{cases}\\
&\;\;\text{and}\;\; (g \cdot x)_h = x_{g^{-1} h} \; .
\end{split}
\end{equation}
We always assume that $\nu$ and $\eta$ are not concentrated on a single atom, because otherwise $(X,\mu)$ consists of a single point. By \cite[Proposition 7.1]{VW17}, the action $\Gamma \actson (X,\mu)$ is essentially free and ergodic, of type III whenever $\nu \neq \eta$. By \cite[Section 7]{VW17}, this family of Bernoulli actions is rich: the crossed products can be of any possible type III$_\lambda$, $\lambda \in (0,1]$, and they can have basically any possible Connes $\tau$-invariant (in the sense of \cite{Con74}).

The main goal of this paper is to prove classification results for the crossed product factors $M = L^\infty(X) \rtimes \Gamma$ given by \eqref{eq.my-bernoulli}. We associate a measure class on the real line to the factor $M$ and prove that it is an isomorphism invariant for this family of type III factors. To formulate this first main result, we introduce some notation.

For every measure class $\mu$ on $\R$, we denote by $\mutil$ the measure class defined by $\mutil(\cU) = 0$ iff $\mu(-\cU) = 0$. We say that a measure class $\mu$ on $\R$ is \emph{stable} if $\delta_0 \prec \mu$, $\mutil \sim \mu$ and $\mu \ast \mu \sim \mu$. For every measure class $\mu$ on $\R$, there is a smallest stable measure class $\gamma$ such that $\mu \prec \gamma$. We denote this as $\gamma = \stc(\mu)$. This measure class $\gamma$ can be defined as the join of the measure classes $\mu^{\ast n} \ast \mutil^{\ast m}$, $n,m \geq 0$, where we use the convention that the $0$'th convolution power is $\delta_0$. Given equivalent probability measures $\nu \sim \eta$ on a standard Borel space $Y$, we consider the stable measure class
\begin{equation}\label{eq.measure-class-gamma}
\gamma = \stc\bigl((\log d\nu / d\eta)_*(\nu)\bigr) \; .
\end{equation}

\begin{letterthm}\label{thm.main}
For $i = 1,2$, let $\Lambda_i$ be a nonamenable group and let $\nu_i \sim \eta_i$ be equivalent probability measures on standard Borel spaces $Y_i$. Consider the nonsingular Bernoulli actions of $\Gamma_i = \Z \ast \Lambda_i$ on $(X_i,\mu_i)$ given by \eqref{eq.my-bernoulli}. Denote by $M_i$ their crossed product von Neumann algebras and let $\gamma_i = \stc\bigl((\log d\nu_i / d\eta_i)_*(\nu_i)\bigr)$ be the stable measure class defined by \eqref{eq.measure-class-gamma}.

If $M_1 \cong M_2$, then $\gamma_1 \sim \gamma_2$.
\end{letterthm}

We also analyze which conclusions can be drawn if $M_1$ merely embeds with expectation into $M_2$, meaning that there exists a faithful normal $*$-homomorphism $\pi : M_1 \to M_2$ and a faithful normal conditional expectation of $M_2$ onto $\pi(M_1)$. The following then provides large families of Bernoulli crossed products for which such embeddings with expectation do not exist. While a systematic study of embeddability between Bernoulli crossed products has been made in \cite{PV21} in the probability measure preserving type II$_1$ setting, our Theorem \ref{thm.main2} is the first such systematic nonembeddability result in the type III case.

To formulate this result, we provide the following canonical class of examples of \eqref{eq.my-bernoulli}. Define the set $\cP$ of Borel probability measures on $\R$ by
\begin{equation}\label{eq.setP}
\cP = \bigl\{ \nu \bigm| \nu \;\;\text{is a Borel probability measure on $\R$ with}\;\; \int_\R \exp(-x) \, d\nu(x) < +\infty \bigr\} \; .
\end{equation}
Given $\nu \in \cP$, there is a unique probability measure $\eta$ on $\R$ given by normalizing $\exp(-x) \, d\nu(x)$. By construction, the measure $(\log d\nu / d\eta)_*(\nu)$ is a translate of $\nu$. Then, \eqref{eq.my-bernoulli} provides a nonsingular Bernoulli action for any free product $\Z \ast \Lambda$, with base space $Y = \R$.

Recall that a Borel set $K \subset \R$ is called independent if every set of $n$ distinct elements of $K$ generates a free abelian subgroup of $\R$ of rank $n$.

\begin{letterthm}\label{thm.main2}
Let $K \subset \R$ be an independent Borel set. For $i = 1,2$, let $\Lambda_i$ be nonamenable groups, put $\Gamma_i = \Z \ast \Lambda_i$ and let $\nu_i \in \cP$ be nonatomic measures supported on $K$. Consider the associated nonsingular Bernoulli actions with crossed product von Neumann algebra $M_i$.

If $M_1$ embeds with expectation into $M_2$, then $\nu_1 \prec \nu_2$. In particular, if $M_1 \cong M_2$, then $\nu_1 \sim \nu_2$.
\end{letterthm}

So, Theorem \ref{thm.main2} provides large classes of nonsingular Bernoulli crossed products that cannot be embedded with expectation one into the other. Moreover, the conclusions of Theorem \ref{thm.main2} hold for further classes of probability measures $\nu_i \in \cP$, see Corollary \ref{cor.Bernoulli-independent-set} below. In Example \ref{exam.full-with-standard-tau}, we use this result to provide mutually nonembeddable type~III Bernoulli crossed products that cannot be distinguished by invariants from modular theory, like Connes $\tau$-invariant.

We next focus on solidity of nonsingular Bernoulli crossed products. Recall from \cite{Oza03} that a II$_1$ factor $M$ is called \emph{solid} if $A' \cap M$ is amenable for every diffuse von Neumann subalgebra $A \subset M$. In \cite{Oza03}, it is proven that the group von Neumann algebra $L(\Gamma)$ is solid for every word hyperbolic group $\Gamma$ and, more conceptually, for every \emph{biexact} countable group $\Gamma$ (see \cite[Chapter 15]{BO08}).
An arbitrary diffuse von Neumann algebra $M$ is called solid if $A' \cap M$ is amenable for every diffuse von Neumann subalgebra $A \subset M$ that is the range of a faithful normal conditional expectation (see \cite{VV05}).
One of the striking features of solid factors is that they are prime: they do not admit nontrivial tensor product decompositions. Also all their nonamenable subfactors with expectation are prime.

Solidity has a counterpart in ergodic theory, as discovered in \cite{CI08}: an essentially free nonsingular action of a countable group $\Gamma$ on a standard nonatomic probability space $(X,\mu)$ is said to be a \emph{solid action} if for every subequivalence relation $\cS$ of the orbit equivalence relation $\cR(\Gamma \actson X)$, there exists a partition of $(X,\mu)$ into $\cS$-invariant Borel sets $(X_n)_{n \geq 0}$ such that $\cS|_{X_0}$ is amenable and $\cS|_{X_n}$ is ergodic for every $n \geq 1$. Note that $\Gamma \actson (X,\mu)$ is a solid action if and only if for every diffuse von Neumann subalgebra $A \subset L^\infty(X)$, the relative commutant $A' \cap L^\infty(X) \rtimes \Gamma$ is amenable. In \cite{CI08}, it is proven that all pmp Bernoulli actions $\Gamma \actson (X_0,\mu_0)^\Gamma$ are solid actions.

It remains an open question whether all nonsingular Bernoulli actions $\Gamma \actson \prod_{h \in \Gamma} (X_0,\mu_h)$ are solid actions. In \cite{HIK20}, it is proven that this is indeed the case when $X_0 = \{0,1\}$ consists of two points and when the probability measures $(\mu_h)_{h \in \Gamma}$ have a stronger almost invariance property: for all $g \in \Gamma$, we have that $\mu_{gh} = \mu_h$ for all but finitely many $h \in \Gamma$. We prove that all nonsingular Bernoulli actions in \eqref{eq.my-bernoulli} are solid actions. Note that our family mainly consists of Bernoulli actions with a diffuse base space, thus complementing the results of \cite{HIK20}.

Actually, our method to prove Theorems \ref{thm.main} and \ref{thm.main2} is a ``solidity method'' that was introduced in \cite{HSV16}. In \cite{HSV16}, it was proven that any faithful normal state $\psi$ on a free Araki-Woods factor $M$ with the property that the centralizer $M^\psi$ is nonamenable, must have a corner that is unitarily conjugate to a corner of the canonical free quasi-free state $\vphi$ on $M$. So in cases where the centralizer of the free quasi-free state $\vphi$ is a nonamenable II$_1$ factor, we can characterize $\vphi$ as the essentially unique state on $M$ having a nonamenable centralizer. As a consequence, the spectral measure class of the modular operator $\Delta_\vphi$ becomes an invariant of such von Neumann algebras $M$.

We thus introduce the following terminology: we say that a faithful normal state $\vphi$ on a von Neumann algebra $M$ is a \emph{solid state} if every faithful normal state $\psi$ on $M$ with a nonamenable centralizer $M^\psi$ has a corner that is unitarily conjugate to a corner of $\vphi$ (see Definition \ref{def.solid-state}). In particular, if $\vphi$ is a solid state on a type III factor and $M^\vphi$ is amenable, it follows that every faithful normal state on $M$ has an amenable centralizer. The main result of \cite{HSV16} can then be reformulated as saying that the free quasi-free state on a free Araki-Woods factor is a solid state.

For our nonsingular Bernoulli actions in \eqref{eq.my-bernoulli} with crossed product $M = L^\infty(X,\mu) \rtimes \Gamma$, it is in general not true that the crossed product state $\vphi_\mu$ is solid. Nevertheless, our proof of Theorems \ref{thm.main} and \ref{thm.main2} is based on carefully analyzing which states on $M$ have a nonamenable centralizer. Under extra assumptions, we do find that $\vphi_\mu$ is a solid state. Our solidity results can then be summarized as follows.

\begin{letterthm}\label{thm.main-solid}
Let $\Lambda$ be a nonamenable group and let $\nu \sim \eta$ be equivalent probability measures on a standard Borel space $Y$. Consider the nonsingular Bernoulli actions of $\Gamma = \Z \ast \Lambda$ on $(X,\mu)$ given by \eqref{eq.my-bernoulli}. Denote $M = L^\infty(X,\mu) \rtimes \Gamma$.
\begin{enumlist}
\item The nonsingular Bernoulli action $\Gamma \actson (X,\mu)$ is a solid action.
\item The factor $M$ is solid relative to $L(\Lambda)$ in the sense of \cite[Definition 3.2]{Mar16}.
\item If $\Lambda$ is biexact, then $M$ is solid.
\item If $\Lambda$ is biexact and $(\log d\nu/d\eta)_*(\nu)$ is nonatomic, then the crossed product state $\vphi_\mu$ on $M$ is a solid state.
\end{enumlist}
\end{letterthm}

In \cite[Corollary 4.5]{Oza04}, it was proven that for every pmp Bernoulli action $\Gamma \actson (X,\mu) = (X_0,\mu_0)^\Gamma$ of a biexact group $\Gamma$, the crossed product $M = L^\infty(X,\mu) \rtimes \Gamma$ is solid. It is an open problem whether the same holds for arbitrary nonsingular Bernoulli actions of biexact groups. By \cite[Theorem C]{HV12}, this problem is equivalent to the open problem whether every nonsingular Bernoulli action of a biexact group is a solid action.

Since we expect that these open problems have a positive solution, it is tempting to believe that for any nonsingular Bernoulli action $\Gamma \actson (X,\mu)$ of a biexact group, the crossed product state $\vphi_\mu$ on $M = L^\infty(X,\mu) \rtimes \Gamma$ is a solid state. This is however not true, as we show in Example \ref{ex.not-solid}. In an attempt to give a more conceptual explanation for Theorem \ref{thm.main}, it is equally natural to try to prove the following statement: if $\Gamma \actson (X,\mu)$ is any nonsingular Bernoulli action with the property that the measure $\mu$ is $\Lambda$-invariant for a nonamenable subgroup $\Lambda < \Gamma$, then the spectral measure class of $\Delta_{\vphi_\mu}$ can be recovered as an invariant of $M$. But the same Example \ref{ex.not-solid} shows that also this statement is false. This explains why our Theorem \ref{thm.main} is restricted to the natural family of actions introduced in \eqref{eq.my-bernoulli}.

We finally prove the following partial converse to Theorem \ref{thm.main}.

\begin{letterprop}\label{prop.some-isomorphism}
Let $\Lambda$ be a countable group and put $\Gamma = \Z \ast \Lambda$. For $i=1,2$, let $\nu_i \sim \eta_i$ be equivalent probability measures on the standard Borel spaces $Y_i$. Denote by $\Gamma \actson^{\al_i} (X_i,\mu_i)$ the associated nonsingular Bernoulli actions given by \eqref{eq.my-bernoulli}. Denote $\sigma_i = (\log d\nu_i / d\eta_i)_*(\nu_i)$.

If $\sigma_1 = \sigma_2$ and if the maps $\log d\nu_i / d\eta_i : (Y_i,\nu_i) \to \R$ are not essentially one-to-one, then there exists a measure preserving conjugacy between the actions $\Gamma \actson^{\al_i} (X_i,\mu_i)$. In particular, the crossed product factors $M_i = L^\infty(X_i,\mu_i) \rtimes_{\al_i} \Gamma$ are isomorphic.
\end{letterprop}

It is clear that Proposition \ref{prop.some-isomorphism} is not an optimal result. One might for instance speculate that the assumption $\sigma_1 \sim \sigma_2$ should be sufficient to prove that the nonsingular Bernoulli actions $\al_i$ are orbit equivalent. Still, our result is nonempty: in Example \ref{ex.some-isomorphism}, we provide examples where the hypotheses of Proposition \ref{prop.some-isomorphism} are satisfied with $Y_1$ being a finite set with atomic measures and $Y_2 = [0,1]$ with two measures $\nu_2 \sim \eta_2$ that are equivalent with the Lebesgue measure. In these examples, there is no obvious conjugacy between the nonsingular Bernoulli actions given by \eqref{eq.my-bernoulli}.

\section{Preliminaries}

A von Neumann subalgebra $B \subset N$ is said to be \emph{with expectation} if there exists a faithful normal conditional expectation $E : N \to B$.

We start by recalling Popa's theory of \emph{intertwining-by-bimodules}, as introduced in \cite[Section 2]{Pop03}. We make use of the adaptations to the semifinite and infinite setting, which reached a final version in \cite[Section 4]{HI15}. So, let $M$ be any von Neumann algebra with separable predual and let $p,q \in M$ be nonzero projections. Let $A \subset pMp$ and $B \subset qMq$ be von Neumann subalgebras with expectation. We write $A \prec_M B$ if there exist projections $r \in A$, $s \in B$, a nonzero partial isometry $v \in r M s$ and a unital normal $*$-homomorphism $\theta : r A r \to s B s$ such that $a v = v \theta(a)$ for all $a \in r A r$ and $\theta(r A r) \subset s B s$ is with expectation.

We write $A \prec_{f,M} B$ if for every nonzero projection $e \in A' \cap pMp$, we have that $Ae \prec_M B$.


When $A$ is finite, $B$ is semifinite, $E_B : qMq \to B$ is a faithful normal conditional expectation and $\Tr$ is a faithful normal semifinite trace on $B$, the following results are contained in \cite[Theorem 4.3]{HI15}.
\begin{itemlist}
\item $A \not\prec_M B$ if and only if there exists a sequence of unitaries $a_n \in \cU(A)$ such that
$$\|E_B(x^* a_n y)\|_{2,\Tr} \to 0 \quad\text{for all $x,y \in p M q$ with $\Tr(x^*x) , \Tr(y^* y) < +\infty$.}$$
\item $A \prec_M B$ if and only if there exists an integer $n \in \N$, a finite projection $s \in M_n(\C) \ot B$, a nonzero partial isometry $v \in (\C^n \ot p M) s$ and a normal unital $*$-homomorphism $\theta : A \to s(M_n(\C) \ot B)s$ such that $a v = v \theta(a)$ for all $a \in A$.
\end{itemlist}

Recall that given a von Neumann subalgebra $A \subset N$, one defines $\cN_N(A) = \{u \in \cU(N) \mid u A u^* = A\}$ and one calls $\cN_N(A)\dpr$ the \emph{normalizer} of $A$ inside $N$. Note that $A' \cap N \subset \cN_N(A)\dpr$. When $A \subset N$ is with expectation, also $\cN_N(A)\dpr \subset N$ is with expectation.

Assume again that $M$ is a von Neumann algebra with separable predual and that $A \subset pMp$ and $B \subset qMq$ are von Neumann subalgebras with expectation. When  $e \in A' \cap pMp$ is a nonzero projection such that $Ae \prec_M B$, we can take a nonzero partial isometry $v$ as above, where $v \in rMs$ and $r \in Ae$. When $u \in \cN_{pMp}(A)$, we can replace $v$ by $uv$ and replace $\theta$ by $\theta \circ \Ad u^*$. It follows that $A \, ueu^* \prec_M B$. We conclude from this argument that there exists a unique projection $z$ in the center $\cZ(\cN_{pMp}(A)\dpr)$ of the normalizer such that $Az \prec_{f,M} B$ and $A(p-z) \not\prec_M B$.


If $M$ is a von Neumann algebra with separable predual and if $B \subset M$ is a von Neumann subalgebra with expectation, then $M$ is \emph{solid relative to} $B$ in the sense of \cite[Definition 3.2]{Mar16} if and only if every von Neumann subalgebra $Q \subset pMp$ with expectation and with diffuse center $\cZ(Q)$ satisfies at least one of the following properties: $Q$ is amenable or $Q \prec_M B$.

For every von Neumann algebra $M$ with separable predual, we denote by $\core(M)$ its \emph{continuous core}, which can be concretely realized as $M \rtimes_{\si^\vphi} \R$ whenever $\vphi$ is a faithful normal state on $M$ with modular automorphism group $(\si^\vphi_t)_{t \in \R}$. We denote by $\lambda_\vphi(t)$, $t \in \R$, the canonical unitary operators in the crossed product $\core(M) = M \rtimes_{\si^\vphi} \R$, generating the von Neumann subalgebra $L_\vphi(\R) \subset \core(M)$. There is a canonical faithful normal semifinite trace $\Tr$ on $\core(M)$. Both the inclusion $M \subset \core(M)$ and the trace $\Tr$ are essentially independent of the choice of $\vphi$, since Connes cocycle derivative theorem provides a trace preserving $*$-isomorphism $\theta : M \rtimes_{\si^\vphi} \R \to M \rtimes_{\si^\om} \R$ satisfying $\theta(a) = a$ for all $a \in M$ and $\theta(\lambda_\vphi(t)) = [D\vphi:D\om]_t \, \lambda_\om(t)$. The restriction of the trace $\Tr$ to $L_\vphi(\R)$ is semifinite. The unique trace preserving conditional expectation $E_{L_\vphi(\R)} : \core(M) \to L_\vphi(\R)$ satisfies $E_{L_\vphi(\R)}(a) = \vphi(a) 1$ for all $a \in M$.

Whenever $P \subset M$ is a von Neumann subalgebra and $E : M \to P$ is a faithful normal conditional expectation, we obtain a canonical trace preserving embedding $\core(P) \hookrightarrow \core(M)$, which can be concretely constructed by taking a faithful normal state $\vphi$ on $P$ and writing $\core(P) = P \rtimes_{\si^\vphi} \R \hookrightarrow M \rtimes_{\si^{\vphi \circ E}} \R = \core(M)$. Note that this embedding depends on the choice of $E$. In the trivial case where $P = \C 1$, we have that $E(a) = \psi(a)1$ and the embedding corresponds to $L_\psi(\R) \subset \core(M)$.

Given an action $\Gamma \actson I$ of a countable group $\Gamma$ on a countable set $I$ and given a von Neumann algebra $(P,\om)$ equipped with a faithful normal state, we consider the generalized Bernoulli action $\Gamma \actson (N,\om) = (P,\om)^I$. Here we use the notation $(P,\om)^I$ to denote the tensor product of copies of $(P,\om)$ indexed by $I$. The action $\Gamma \actson (N,\om)$ is state preserving. We get a canonical action of $\Gamma$ on the continuous core $\core(N)$ such that
$$\core(N \rtimes \Gamma) = \core(N) \rtimes \Gamma \; .$$

In \cite{Pop03}, Popa introduced his fundamental malleable deformation for probability measure preserving Bernoulli actions $\Gamma \actson (X_0,\mu_0)^\Gamma$, which has been a cornerstone for deformation/\allowbreak rigidity theory. It has been extended in several directions. In \cite{Ioa06}, another malleable deformation was found, adapted to noncommutative Bernoulli actions $\Gamma \actson (P,\tau)^\Gamma$, where $(P,\tau)$ is a tracial von Neumann algebra. This can be adapted in a straightforward way to the nontracial case, i.e.\ for Bernoulli actions $\Gamma \actson (P,\om)^\Gamma$, where $\om$ is a faithful normal state on $P$ (see e.g.\ \cite[Section 5]{Mar16}). Also Popa's spectral gap rigidity for Bernoulli actions, as introduced in \cite{Pop06}, can be extended to the setting of generalized Bernoulli actions, i.e.\ for actions $\Gamma \actson (X_0,\mu_0)^I$, where $I$ is a countable set on which $\Gamma$ is acting, see \cite[Section 4]{IPV10}. Putting all these generalizations together, we right away get the following variant of \cite[Corollary 4.3]{IPV10}.

\begin{theorem}\label{thm.def-rig-nc-bernoulli}
Let $(P,\om)$ be an amenable von Neumann algebra with a faithful normal state. Let $\Gamma \actson I$ be an action of a countable group $\Gamma$ on a countable set $I$. Assume that $\Stab(i)$ is amenable for every $i \in I$ and assume that there exists a $\kappa \in \N$ such that $\Stab J$ is finite whenever $J \subset I$ and $|J| \geq \kappa$. Denote, as above, $(N,\om) = (P,\om)^I$ and let $\Gamma \actson (N,\om)$ be the generalized Bernoulli action. Write $M = N \rtimes \Gamma$.

Let $p \in \core(M)$ be a projection of finite trace and $A \subset p \core(M) p$ a von Neumann subalgebra such that the relative commutant $A' \cap p \core(M) p$ has no amenable direct summand. Denote by $Q = \cN_{p \core(M) p}(A)\dpr$ the normalizer of $A$. Let $z \in \cZ(Q)$ be the maximal projection such that $A z \prec_f L_\om(\R)$. Put $z' = p-z$. Then $Q z' \prec_f L_\om(\R) \vee L(\Gamma)$.
\end{theorem}
\begin{proof}
Write $\core(L(\Gamma)) = L_\om(\R) \vee L(\Gamma)$. Replacing $A$ by $A z\dpr$ where $z\dpr$ is an arbitrary projection in $\cZ(Q) z'$, we may assume that $A \not\prec L_\om(\R)$ and we have to prove that $Q \prec \core(L(\Gamma))$. Even though $\om$ is not necessarily tracial, the tensor length deformation makes sense in this context and the proof of \cite[Corollary 4.3]{IPV10} can be copied almost verbatim. The conclusion is that at least one of the following statements hold: $Q \prec \core(N) \rtimes \Stab i$ for some $i \in I$, or $Q \prec \core(L(\Gamma))$. The von Neumann algebra $\core(N) \rtimes \Stab i$ is amenable. Since $A' \cap p \core(M) p \subset Q$, the von Neumann algebra $Q$ has no amenable direct summand. Therefore, it is impossible that $Q \prec \core(N) \rtimes \Stab i$. This concludes the proof of the theorem.
\end{proof}

Also the following result is an immediate noncommutative variant of known results for probability measure preserving Bernoulli actions. The method was introduced in \cite[Section 3]{Pop03} and the following version is a straightforward generalization of \cite[Lemma 4.2]{Vae07}. The same result still holds when replacing the ad hoc construction \eqref{eq.pseudo-almost} by the quasinormalizer of $B$ inside $pMp$, but we only need this simpler version.

\begin{proposition}\label{prop.control-normalizer}
Make the same assumptions as in Theorem \ref{thm.def-rig-nc-bernoulli}. Let $p \in L(\Gamma)$ be a projection and $B \subset p L(\Gamma) p$ a diffuse von Neumann subalgebra. Define
\begin{equation}\label{eq.pseudo-almost}
D = \bigl\{u \in p M p \bigm| \exists \be \in \Aut(B) \; , \; \forall b \in B \; : \; u b = \be(b) u \bigr\}\dpr
\end{equation}
and note that $\cN_{pMp}(B)\dpr \subset D$.
\begin{enumlist}
\item If $B \not\prec_{L(\Gamma)} L(\Stab i)$ for every $i \in I$, then $D \subset p L(\Gamma) p$.
\item If $rDr$ is nonamenable for every nonzero projection $r \in B' \cap p L(\Gamma) p$, then $D \subset p L(\Gamma) p$.
\end{enumlist}
\end{proposition}
\begin{proof}
Since $L(\Gamma)$ lies in the centralizer of the state $\om$ on $M$, all computations of \cite[Lemma 4.2]{Vae07} go through verbatim. So, if $B \not\prec_{L(\Gamma)} L(\Stab i)$ for every $i \in I$, it follows from \cite[Lemma 4.2]{Vae07} that $D \subset p L(\Gamma) p$.

Next assume that there exists an $i \in I$ such that $B \prec_{L(\Gamma)} L(\Stab i)$. It suffices to prove that $rDr$ is amenable for some nonzero projection $r \in B' \cap p L(\Gamma) p$. By assumption, $\Stab J$ is finite whenever $J \subset I$ and $|J| \geq \kappa$. Also, $B$ is diffuse. We thus find a finite nonempty subset $J \subset I$ such that $B \prec_{L(\Gamma)} L(\Stab J)$ and $B \not\prec_{L(\Gamma)} L(\Stab(J \cup \{j\}))$ for every $j \in I \setminus J$.

As in \cite[Remark 3.8]{Vae07}, we can take an integer $n \in \N$, a projection $q \in M_n(\C) \ot L(\Stab J)$, a nonzero partial isometry $v \in (\C^n \ot p L(\Gamma))q$ and a unital normal $*$-homomorphism $\theta : B \to q (M_n(\C) \ot L(\Gamma)) q$ such that $b v = v \theta(b)$ for all $b \in B$ and such that $\theta(B) \not\prec_{L(\Stab J)} L(\Stab(J \cup \{j\}))$ for every $j \in I \setminus J$. When $u \in pMp$ and if $\be \in \Aut(B)$ such that $u b = \be(b) u$ for all $b \in B$, it follows that
$$v^* u v \, \theta(b) = \theta(\be(b)) \, v^* u v \quad\text{for all $b \in B$.}$$
By  \cite[Lemma 4.2]{Vae07}, it follows that $v^* u v \in N \rtimes \Norm J$, where $\Norm J = \{g \in \Gamma \mid g \cdot J = J \}$. Writing $r = v v^*$ and $s = v^* v$, we get that $r$ is a projection in $B' \cap p L(\Gamma) p \subset D$, that $s$ is a projection in $N \rtimes \Norm J$ and that $v^* D v \subset N \rtimes \Norm J$. Since $L(\Gamma) \subset N \rtimes \Gamma$ is with expectation, also $B \subset p M p$ and thus $D \subset p M p$ are with expectation. It follows that $v^* D v$ is with expectation in $s(N \rtimes \Norm J)s$. Since $J$ is finite and nonempty, the group $\Norm J$ is amenable. We conclude that $v^* D v$ is amenable, so that $rDr$ is amenable.
\end{proof}


\section{Measure classes of faithful normal states}\label{sec.measure-classes}

For any self-adjoint, possibly unbounded operator $T$, we denote by $\class(T)$ its spectral measure class on $\R$. Note that a Borel set $\cU \subset \R$ has measure zero for $\class(T)$ if and only if the spectral projection $1_\cU(T)$ equals $0$.

Given a faithful normal state $\om$ on a von Neumann algebra $M$, we define the measure class $\class(\om) := \class(\log \Delta_\om)$, where $\Delta_\om$ is the modular operator of $\om$. Of course, $\class(\om)$ highly depends on the choice of the state $\om$ and hence, does not provide an invariant of the von Neumann algebra $M$. A key element of this paper is that certain von Neumann algebras, including many nonsingular Bernoulli crossed products, have a favorite state $\om$ that can be essentially intrinsically characterized, so that $\class(\om)$ becomes an isomorphism invariant for this family of von Neumann algebras.

To establish these results, we rephrase \cite[Corollary 3.2]{HSV16} in the following way, also introducing the notation $\vphi \prec_f \om$ for faithful normal states on a von Neumann algebras.

\begin{lemma}[{\cite[Corollary 3.2]{HSV16}}]\label{lem.HSV}
Let $\vphi$ and $\om$ be faithful normal states on a von Neumann algebra $M$, with corresponding canonical subalgebras $L_\vphi(\R)$ and $L_\om(\R)$ of the continuous core $\core(M)$. Then the following three statements are equivalent.
\begin{enumlist}
\item $L_\vphi(\R) \prec_{\core(M)} L_\om(\R)$.
\item There exist a nonzero partial isometry $v \in M$ and $\gamma > 0$ such that $\gamma^{it} \, [D\om:D\vphi]_t \, \si^\vphi_t(v) = v$ for all $t \in \R$.
\item There exists a nonzero partial isometry $v \in M$ with $e := v^* v \in M^\vphi$, $q:= vv^* \in M^\om$ and $\vphi(e)^{-1} \vphi(v^* x v) = \om(q)^{-1} \om(x)$ for all $x \in qMq$.
\end{enumlist}
When these equivalent conditions hold, we write $\vphi \prec \om$. Note that by 3, we have $\vphi \prec \om$ iff $\om \prec \vphi$.

Also the following three statements are equivalent.
\begin{enumlist}[resume]
\item $L_\vphi(\R) \prec_{f,\core(M)} L_\om(\R)$.
\item For every nonzero projection $p \in M^\vphi$, there exist $\gamma > 0$ and $v \in M$ as in 2 with $v^* v \leq p$.
\item For every nonzero projection $p \in M^\vphi$, there exists $v \in M$ as in 3 with $v^* v \leq p$.
\end{enumlist}
When these equivalent conditions hold, we write $\vphi \prec_f \om$.
\end{lemma}

%

In full generality, the relation $\vphi \prec_f \om$ is not strong enough to conclude that $\class(\vphi) \prec \class(\om)$. We nevertheless have the following partial results.

\begin{proposition}\label{prop.class}
Let $\vphi$ and $\om$ be faithful normal states on a von Neumann algebra $M$.
\begin{enumlist}
\item If $\vphi \prec_f \om$, there exists an atomic probability measure $\rho$ on $\R$ such that $\class(\vphi) \prec \rho \ast \class(\om)$.
\item If $\vphi \prec \om$ and if $M^\vphi$ is a factor, then $\class(\vphi) \prec \class(\om)$.
\end{enumlist}
\end{proposition}
\begin{proof}
1.\ We apply point 5 of Lemma \ref{lem.HSV}. We thus find a sequence of nonzero projections $p_n \in M^\vphi$, partial isometries $v_n \in M$ and $\gamma_n > 0$ such that $\sum_n p_n = 1$, $v_n^* v_n = p_n$ and
$$\gamma_n^{it} \, [D\om:D\vphi]_t \, \si^\vphi_t(v_n) = v_n \quad\text{for all $n$ and all $t \in \R$.}$$
Define $H = \ell^2(\N^2) \ot L^2(M,\om)$ and consider the unitary representation
$$\theta : \R \to \cU(H) : \theta(t)(\delta_{n,m} \ot \xi) = (\gamma_n / \gamma_m)^{it} \, \delta_{n,m} \ot \Delta_\om^{it} \xi \; .$$
Then,
$$V : L^2(M,\vphi) \to H : V(x) = \sum_{n,m} \gamma_m^{1/2} \, \delta_{n,m} \ot v_n x v_m^*$$
is a well defined isometry satisfying $\theta(t) V = V \Delta_\vphi^{it}$ for all $t \in \R$. Hence, $(\Delta_\vphi^{it})_{t \in \R}$ is unitarily equivalent with a subrepresentation of $\theta$. Choosing an atomic probability measure $\rho$ on $\R$ with atoms in the points $\log \gamma_n - \log \gamma_m$, it follows that $\class(\vphi) \prec \rho \ast \class(\om)$.

2.\ We apply point 2 of Lemma \ref{lem.HSV}. Take a  nonzero projections $p \in M^\vphi$, a partial isometry $v \in M$ and $\gamma > 0$ such that $v^* v = p$ and $\gamma^{it} \, [D\om:D\vphi]_t \, \si^\vphi_t(v) = v$ for all $t \in \R$. Since $M^\vphi$ is a factor, we can choose partial isometries $w_n \in M^\vphi$ with $w_n w_n^* \leq p$ and $\sum_n w_n^* w_n = 1$. We can then apply the proof of the first point to the partial isometries $v w_n$, with $\gamma_n = \gamma$ for all $n$. The conclusion then becomes $\class(\vphi) \prec \class(\om)$.
\end{proof}


For later purposes, we prove the following rather specific and technical variant of the second point in Proposition \ref{prop.class}.

\begin{lemma}\label{lem.class}
Let $N$ be a von Neumann algebra with von Neumann subalgebra $M \subset N$ and faithful normal conditional expectation $E : N \to M$. Let $\om_0$ and $\vphi_0$ be faithful normal states on $M$. Write $\om = \om_0 \circ E$ and $\vphi = \vphi_0 \circ E$. Assume that there exists a subset $\cG \subset \cU(N)$ such that $u^* \si^\om_t(u) \in M$ for all $t \in \R$, $u \in \cG$ and such that the linear span of $\cG M$ is dense in $L^2(N,\om)$.

If $\vphi \prec \om$ and if $M^{\vphi_0}$ is a factor, there exists $u \in \cG$ such that
$$\class(\vphi_0) \prec \class((\om \circ \Ad u)|_M) \prec \class(\om) \; .$$
\end{lemma}

\begin{proof}
By Lemma \ref{lem.HSV}, we find a nonzero element $v \in N$ and $\gamma > 0$ such that
\begin{equation}\label{eq.my-eq1}
\gamma^{it} \, \si^\om_t(v) \, [D\om:D\vphi]_t = \gamma^{it} \, [D\om:D\vphi]_t \, \si^\vphi_t(v) = v \quad\text{for all $t \in \R$.}
\end{equation}
Since the linear span of $\cG M$ is dense in $L^2(N,\om)$, we can choose $u \in \cG$ such that $E(u^* v) \neq 0$.

Denote $a_t := u^* \si^\om_t(u) \in \cU(M)$. Define the faithful normal state $\psi = \om \circ \Ad u$ on $N$. By construction, $[D\psi : D \om]_t = a_t$. Since $a_t \in M$, we get that $\psi = \psi_0 \circ E$, where $\psi_0$ is defined as the restriction of $\psi$ to $M$. We have to prove that $\class(\vphi_0) \prec \class(\psi_0)$.

For every $x \in N$ and $t \in \R$, we have
\begin{align*}
\si_t^\psi(x) \, [D \psi : D \vphi]_t &= u^* \si^\om_t(u x u^*) u \, [D \psi : D \vphi]_t = a_t \, \si^\om_t(x) \, [D\om : D\psi]_t \, [D \psi : D \vphi]_t \\ &= a_t \, \si_t^\om(x) \, [D \om : D \vphi]_t \; .
\end{align*}
Applying this to $x = E(u^* v)$, using \eqref{eq.my-eq1} and using that $[D \psi : D \vphi]_t = [D\psi_0 : D\vphi_0]_t \in M$ and $[D \om : D \vphi]_t = [D \om_0 : D\vphi_0]_t \in M$, we get that
\begin{align*}
\gamma^{it} \, \si_t^{\psi_0}(x) \, [D \psi_0 : D \vphi_0]_t & = \gamma^{it} \, a_t \, \si_t^\om(E(u^* v)) \, [D \om : D \vphi]_t = a_t \, E\bigl( \si_t^\om(u)^* \, \gamma^{it} \,  \si_t^\om(v) \, [D \om : D \vphi]_t\bigr)\\
& = a_t \, E(a_t^* \, u^* v) = E(u^* v) = x \; .
\end{align*}
Defining the partial isometry $w \in M$ as the polar part of $x$, we still have
$$\gamma^{it} \, \si_t^{\psi_0}(w) \, [D \psi_0 : D \vphi_0]_t = w \; .$$
Since $M^{\vphi_0}$ is a factor, it then follows from the second point of Proposition \ref{prop.class} that $\class(\vphi_0) \prec \class(\psi_0)$. Since $\psi = \psi_0 \circ E$, we have that $\class(\psi_0) \prec \class(\psi)$. Since $\psi = \om \circ \Ad u$, the modular operators of $\psi$ and $\om$ are unitarily equivalent, so that $\class(\psi) = \class(\om)$.
\end{proof}

\section{Proof of Theorems \ref{thm.main} and \ref{thm.main2}}

We deduce Theorems \ref{thm.main} and \ref{thm.main2} from a more general rigidity result for certain nonsingular coinduced actions.

Let $G$ be a countable amenable group and let $G \actson^\al (Z,\zeta)$ be any nonsingular action on a nontrivial standard probability space. Given any countable infinite group $\Lambda$, we consider the free product $\Gamma = G \ast \Lambda$ and the countable set $I = (G \ast \Lambda)/G$ with base element $i_0 = e G$. We then define $(X,\mu) = (Z,\zeta)^I$ and the nonsingular action $\Gamma \actson (X,\mu)$ by
\begin{equation}\label{eq.more-general-nonsingular}
\begin{split}
&(g \cdot x)_{i} = x_{g^{-1} \cdot i} \;\;\text{when $g \in \Lambda$,}\\
& (g \cdot x)_{i} = x_{g^{-1} \cdot i} \;\;\text{when $i \neq i_0$ and $g \in G$, and}\quad (g \cdot x)_{i_0} = g \cdot x_{i_0}\;\;\text{when $g \in G$.}
\end{split}
\end{equation}
So, $\Lambda$ acts as a generalized Bernoulli action on $X = Z^I$, while the action of $G$ is the diagonal product of a generalized Bernoulli action on $Z^{I \setminus \{i_0\}}$ and the given action $G \actson^\al Z$, viewed as the $i_0$-coordinate.

Note that the nonsingular Bernoulli action in \eqref{eq.my-bernoulli} arises as a special case of \eqref{eq.more-general-nonsingular} by taking $G = \Z$ and $\Z \actson (Z,\zeta) = \prod_{n \in \Z} (Y,\mu_n)$ by Bernoulli shift.

When $i \neq j$ are distinct elements, the stabilizer $\Stab \{i,j\}$ is trivial. It follows that $\Gamma \actson (X,\mu)$ is essentially free. The action of $\Lambda$ is measure preserving and ergodic. It follows that $\Gamma \actson (X,\mu)$ is ergodic, with the Krieger type being determined in the following way: if $\zeta$ is $G$-invariant, then $\Gamma \actson (X,\mu)$ is of type II$_1$. In all other cases, $\Gamma \actson (X,\mu)$ is of type III.

We associate to $G \actson^\al (Z,\zeta)$ the stable measure class $\stc(\al)$ defined as the smallest stable measure class such that $(\log d(g \cdot \zeta)/ d\zeta)_*(\zeta) \prec \stc(\al)$ for all $g \in G$. Note that $\Gamma \actson (X,\mu)$ is of type III$_\lambda$ if and only if $\stc(\al)$ is equivalent with the counting measure on $\Z \log \lambda$. Otherwise, $\Gamma \actson (X,\mu)$ is of type III$_1$.

Theorems \ref{thm.main} and \ref{thm.main2} will be deduced from the following more general result.

\begin{theorem}\label{thm.main3}
For $i \in \{1,2\}$, let $G_i$ be countable amenable groups with nonsingular actions $G_i \actson^{\al_i} (Z_i,\zeta_i)$ on nontrivial standard probability spaces. Let $\Lambda_i$ be nonamenable groups. Put $\Gamma_i = G_i \ast \Lambda_i$ and define $\Gamma_i \actson (X_i,\mu_i)$ by \eqref{eq.more-general-nonsingular}. Write $M_i = L^\infty(X_i,\mu_i) \rtimes \Gamma_i$.
\begin{enumlist}
\item If $M_1 \cong M_2$, then $\stc(\al_1) \sim \stc(\al_2)$.
\item If $M_1$ embeds with expectation into $M_2$, there exists an atomic probability measure $\rho$ on $\R$ such that $\stc(\al_1) \prec \rho \ast \stc(\al_2)$.
\end{enumlist}
\end{theorem}

We prove Theorem \ref{thm.main3} by combining several ingredients: we first prove how the crossed product of an action of the form \eqref{eq.more-general-nonsingular} can be embedded in a state preserving way into a generalized Bernoulli crossed product as studied in Theorem \ref{thm.def-rig-nc-bernoulli}. We then combine Theorem \ref{thm.def-rig-nc-bernoulli} with the results of Section \ref{sec.measure-classes} to reach the conclusions of Theorem \ref{thm.main3}.


\begin{lemma}\label{lem.canonical-embedding}
Let $\Gamma \actson (X,\mu)$ be defined by \eqref{eq.more-general-nonsingular}. Define $P = L^\infty(Z,\zeta) \rtimes G$ and denote by $\om$ the canonical crossed product state on $P$. There is a canonical state preserving embedding $\psi$ of $L^\infty(X,\mu) \rtimes \Gamma$ into the generalized Bernoulli crossed product $(P,\om)^I \rtimes \Gamma$, and there is a state preserving conditional expectation of $(P,\om)^I \rtimes \Gamma$ onto the image of $\psi$.
\end{lemma}
\begin{proof}
We denote by $(u_g)_{g \in \Gamma}$ the canonical unitary operators in the crossed products $L^\infty(X)\rtimes \Gamma$ and $(P,\om)^I \rtimes \Gamma$. We denote by $(v_g)_{g \in G}$ the canonical unitary operators in the crossed product $L^\infty(Z) \rtimes G = P$. We denote by $\pi_0 : (P,\om) \to (P,\om)^I$ the canonical embedding in coordinate $i_0 \in I$. Note that $\pi_0(P)$ commutes with $(u_g)_{g \in G}$ inside $(P,\om)^I \rtimes \Gamma$. We denote by $\vphi$ the crossed product state on $L^\infty(X,\mu) \rtimes \Gamma$. We still denote by $\om$ the natural state on $(P,\om)^I \rtimes \Gamma$.

We can then define the state preserving embedding
$$\psi : L^\infty(X) \rtimes \Gamma \to (P,\om)^I \rtimes \Gamma$$
such that the restriction of $\psi$ to $L^\infty(X)$ is the canonical embedding of $L^\infty(X,\mu) = L^\infty((Z,\zeta)^I)$ into $(P,\om)^I$ and such that
$$\psi(u_g) = u_g \;\;\text{for all $g \in \Lambda$, and}\quad \psi(u_h) = \pi_0(v_h) \, u_h \;\;\text{for all $h \in G$.}$$
By construction, $\psi \circ \si_t^\vphi = \si_t^\om \circ \psi$. There thus exists a state preserving conditional expectation of $(P,\om)^I \rtimes \Gamma$ onto the image of $\psi$.
\end{proof}

By Lemma \ref{lem.canonical-embedding} any embedding with expectation of a crossed product von Neumann algebra $N$ into $L^\infty(X,\mu) \rtimes \Gamma$ will induce an embedding with expectation of $N$ into $(P,\om)^I \rtimes \Gamma$. As a preparation to prove Theorem \ref{thm.main3}, we thus prove a general rigidity result for such embeddings into generalized Bernoulli crossed products $(P,\om)^I \rtimes \Gamma$. We actually prove a very general result of this kind, which is of independent interest.

Let $\cR$ be a nonsingular countable equivalence relation on a standard probability space $(X,\mu)$. In the context of the discussion above, $\cR$ would be the orbit equivalence relation of a nonsingular action of the form \eqref{eq.more-general-nonsingular}, but we prove results for arbitrary equivalence relations $\cR$. We denote by $\vphi_\mu$ the canonical faithful normal state on the von Neumann algebra $L(\cR)$. We assume that the centralizer $L(\cR)^{\vphi_\mu}$ is large, in the sense that it has no amenable direct summand. We prove that if $L(\cR)$ embeds with expectation $E$ into a noncommutative Bernoulli crossed product $(M,\om)$ satisfying the appropriate conditions, then automatically $\vphi_\mu \circ E \prec_f \om$, using the notation of Lemma \ref{lem.HSV}.

Recall that a countable nonsingular equivalence relation $\cR$ on a standard probability space $(X,\mu)$ is said to be purely infinite if for every nonnegligible Borel set $\cU \subset X$, the restriction $\cR|_\cU$ does not admit an $\cR$-invariant probability measure that is equivalent with $\mu$. When $\cR$ is ergodic, this is equivalent to saying that $\cR$ is of type III.


\begin{proposition}\label{prop.main-technical}
Let $\cR$ be a countable nonsingular equivalence relation on the standard probability space $(X,\mu)$. Assume that $\cR$ is purely infinite and that the centralizer $L(\cR)^{\vphi_\mu}$ has no amenable direct summand.

Let $(M,\om)$ be the noncommutative generalized Bernoulli crossed product $M = (P,\om)^I \rtimes \Gamma$, where $P$ is amenable, $\om$ is a faithful normal state on $P$, $\Gamma$ is any countable group and the action $\Gamma \actson I$ has the properties that $\Stab(i)$ is amenable for all $i \in I$ and that there exists a $\kappa \in \N$ such that $\Stab(J)$ is finite whenever $J \subset I$ satisfies $|J| \geq \kappa$.

If $\pi : L(\cR) \to M$ is an embedding of $L(\cR)$ as a von Neumann subalgebra of $M$ admitting a faithful normal conditional expectation $E : M \to \pi(L(\cR))$, then $\vphi_\mu \circ \pi^{-1} \circ E \prec_f \om$.
\end{proposition}
\begin{proof}
We write $\vphi = \vphi_\mu$ and $N = L(\cR)$. We view $N$ as a von Neumann subalgebra of $M$ with the faithful normal conditional expectation $E : M \to N$. We still denote by $\vphi$ the faithful normal state $\vphi \circ E$ on $M$. By Lemma \ref{lem.HSV}, we have to prove that $L_\vphi(\R) \prec_f L_\om(\R)$ inside the continuous core $\core(M)$.

Take an arbitrary nonzero projection $p \in L_\vphi(\R)$ of finite trace and write $A = L_\vphi(\R) p$. We prove that $A \prec_f L_\om(\R)$. Note that $N^\vphi p$ commutes with $A$ and has no amenable direct summand. Denote $Q = A' \cap p \core(M) p$ and let $z \in \cZ(Q)$ be the maximal projection such that $A z \prec_f L_\om(\R)$. Put $z' = p-z$. Assume that $z' \neq 0$. We derive a contradiction. Write $\core(L(\Gamma)) = L_\om(\R) \vee L(\Gamma)$. By Theorem \ref{thm.def-rig-nc-bernoulli}, $Q \prec \core(L(\Gamma))$.

Write $B = (L^\infty(X) \vee L_\vphi(\R)) p$. Since $L^\infty(X)$ commutes with $L_\vphi(\R)$, we have that $B \subset Q$. Thus, $B \prec \core(L(\Gamma))$. We prove that $L^\infty(X) \prec_M L(\Gamma)$. Assume the contrary. Denote by $E_{L(\Gamma)} : M \to L(\Gamma)$ the unique $\om$-preserving conditional expectation. Since $L^\infty(X) \not\prec_M L(\Gamma)$, we can take a sequence of unitaries $w_n \in \cU(L^\infty(X))$ such that $E_{L(\Gamma)}(x^* w_n y) \to 0$ $*$-strongly for all $x,y \in M$. We claim that
$$\bigl\|E_{\core(L(\Gamma))}(x^* w_n  y)\bigr\|_{2,\Tr} \to 0 \quad\text{for all $x,y \in \core(M)$ with $\Tr(x^* x) < +\infty$ and $\Tr(y^* y) < +\infty$.}$$
By density, it suffices to prove this claim for $x = x_1 x_2$ and $y = y_1 y_2$ with $x_1,y_1 \in M$ and $x_2,y_2 \in L_\om(\R)$ with $\Tr(x_2^* x_2) < +\infty$ and $\Tr(y_2^* y_2) < +\infty$. But then,
$$E_{\core(L(\Gamma))}(x^* w_n  y) = x_2^* \, E_{L(\Gamma)}(x_1^* w_n y_1) \, y_2 \; ,$$
so that the claim follows. We find in particular that $\bigl\|E_{\core(L(\Gamma))}(x^* w_n p \,  y)\bigr\|_{2,\Tr} \to 0$ for all $x,y \in \core(M)$. Since $w_n p$ is a sequence of unitaries in $B$, this implies that $B \not\prec \core(L(\Gamma))$, which is a contradiction. So, we have proven that $L^\infty(X) \prec_M L(\Gamma)$.

Take projections $e \in L^\infty(X)$, $q \in L(\Gamma)$, a nonzero partial isometry $v \in e M q$ and a normal unital $*$-homomorphism $\theta : L^\infty(X) e \to q L(\Gamma) q$ such that $a v = v \theta(a)$ for all $a \in L^\infty(X) e$. Write $B_1 = L^\infty(X) e$. Since $v^* v$ commutes with $\theta(B_1)$, also the support projection of $E_{L(\Gamma)}(v^* v)$ commutes with $\theta(B_1)$. We may cut down with this projection and thus assume that the support projection of $E_{L(\Gamma)}(v^* v)$ equals $q$. Next, we may also replace $e$ by the support of the homomorphism $\theta$ and assume that $\theta$ is faithful.

Define $B_2 = \theta(B_1)$ and
$$D_2 = \bigl\{ u \in q M q \bigm| \exists \be \in \Aut(B_2) \; , \; \forall b \in B_2 \; : \; u b = \be(b) u \bigr\}\dpr \; .$$
Define $D_1 = \cN_{eMe}(B_1)\dpr$. Note that $e L(\cR) e \subset D_1$. Whenever $u \in \cU(e M e)$ normalizes $B_1$, we get that $v^* u v \in D_2$. Write $s = v^* v$. Thus, $s \in D_2$ and $v^* D_1 v \subset s D_2 s$. Also, $vv^* \in D_1$. Since $L(\cR)$ has no amenable direct summand and $L(\cR) \subset M$ is with expectation, we conclude that $s D_2 s$ has no amenable direct summand. Let $z \in \cZ(D_2)$ be the central support of $s$ in $D_2$. Then $D_2 z$ has no amenable direct summand. When $r \in B_2' \cap q L(\Gamma)q \subset D_2$ is a nonzero projection, since $q$ is equal to the support projection of $E_{L(\Gamma)}(s)$, we find that $rs \neq 0$. Thus, $rz \neq 0$, so that $r D_2 r$ is nonamenable. Since this holds for every choice of $r$, it follows from Proposition \ref{prop.control-normalizer} that $D_2 \subset q L(\Gamma) q$. In particular, $L(\cR) \prec_M L(\Gamma)$. It follows that $L(\cR)$ has a direct summand that is finite, contradicting our assumption that $\cR$ is purely infinite. This final contradiction shows that $z'=0$. So, $L_\vphi(\R) \prec_f L_\om(\R)$ and the proposition is proven.
\end{proof}

%
%
%

\begin{proof}[{Proof of Theorem \ref{thm.main3}}]
If $G_1 \actson (Z_1,\zeta_1)$ is measure preserving, then $\stc(\al_1) = \delta_0$ and there is nothing to prove. We may thus assume that $G_1 \actson (Z_1,\zeta_1)$ is not measure preserving. As explained above, it follows that $\Gamma \actson (X_1,\mu_1)$ is ergodic and of type III.

Denote by $\vphi_i$ the canonical crossed product state on $M_i$. Assume that $\pi : M_1 \to M_2$ is any embedding with expectation. By Lemma \ref{lem.canonical-embedding}, we view $M_2$ as a von Neumann subalgebra of a generalized Bernoulli crossed product $N_2 = (P_2,\om_2)^{I_2} \rtimes \Gamma_2$, where $I_2 = \Gamma_2 / G_2$ and $P_2 = L^\infty(Z_2,\zeta_2) \rtimes G_2$, with crossed product state $\om_2$ on $P_2$. We still denote by $\om_2$ the natural state on $N_2$. There is a unique faithful normal conditional expectation $E : N_2 \to M_2$ satisfying $\om_2 = \vphi_2 \circ E_2$. The action $\Gamma_2 \actson I_2 = \Gamma_2 / G_2$ has amenable stabilizers and has the property that $\Stab\{i,j\} = \{e\}$ when $i \neq j$.

We apply Proposition \ref{prop.main-technical} to the orbit equivalence relation $\cR_1 = \cR(\Gamma_1 \actson X_1)$, which is ergodic and of type III. Note that $L(\cR_1)$ is canonically isomorphic with $M_1$ and the state $\vphi_{\mu_1}$ in Proposition \ref{prop.main-technical} is equal to the canonical state $\vphi_1$ on the crossed product $M_1 = L^\infty(X_1,\mu_1) \rtimes \Gamma_1$. The centralizer $M_1^{\vphi_1}$ contains $L^\infty(X_1,\mu_1) \rtimes \Lambda_1$, which has trivial relative commutant in $M_1$. So, $M_1^{\vphi_1}$ is a nonamenable factor. By Proposition \ref{prop.main-technical}, we find that $\vphi_1 \circ \pi^{-1} \circ E \prec_f \om_2$. Using Proposition \ref{prop.class}, we find an atomic probability measure $\rho$ such that
$$\class(\vphi_1 \circ \pi^{-1} \circ E) \prec \rho \ast \class(\om_2) \; .$$
Since $\class(\vphi_1) \prec \class(\vphi_1 \circ \pi^{-1} \circ E)$, we get that
\begin{equation}\label{eq.almost-second-statement}
\class(\vphi_1) \prec \rho \ast \class(\om_2) \; .
\end{equation}
We now prove that $\class(\vphi_1) = \stc(\al_1)$ and $\class(\om_2) = \stc(\al_2)$. By construction, for every $g \in \Gamma_1$, the measure $g \cdot \mu_1$ is of the form $g \cdot \mu_1 = \prod_{i \in I} (g_i \cdot \zeta_1)$ with $g_i \in G_1$ and with all but finitely many $g_i$ equal to $e$. Moreover, any collection of such $g_i \in G_1$ can be realized by the appropriate choice of $g \in \Gamma_1$. Since $\class(\vphi_1)$ is the join of the measure classes $(\log d(g \cdot \mu_1)/ d\mu_1)_*(\mu_1)$, $g \in \Gamma_1$, it follows that $\class(\vphi_1)$ is the join of all convolution products of $(\log d(g \cdot \zeta_1)/ d\zeta_1)_*(\zeta_1)$, $g \in G_1$. This is precisely $\stc(\al_1)$.

Secondly, $\class(\om_2)$ is the join of all convolution powers of $\class(\om_2|_{P_2})$. Since $\class(\om_2|_{P_2})$ is the join of the measure classes $(\log d(g \cdot \zeta_2)/ d\zeta_2)_*(\zeta_2)$, $g \in G_2$, it follows that $\class(\om_2) = \stc(\al_2)$.
The second statement of the theorem thus follows from \eqref{eq.almost-second-statement}.

To prove the first statement of the theorem, assume that $\pi$ is a $*$-isomorphism between $M_1$ and $M_2$. By symmetry, it suffices to prove that $\stc(\al_1) \prec \stc(\al_2)$. We apply Lemma \ref{lem.class}. Since $P_2 = L^\infty(Z_2) \rtimes G_2$, we may view
$$N_2 = (P_2,\om_2)^{I_2} \rtimes \Gamma_2 = L^\infty\bigl((Z_2,\zeta_2)^{I_2}\bigr) \rtimes \cG_2 \; ,$$
where $\cG_2 = G_2 \wr_{I_2} \Gamma_2 = G_2^{(I_2)} \rtimes \Gamma_2$ is the generalized wreath product group that acts naturally on $(Z_2,\zeta_2)^{I_2} = (X_2,\mu_2)$. For every $g \in \cG_2$, we have that $u_g^* \si^{\om_2}_t(u_g) \in L^\infty(X_2,\mu_2) \subset M_2 = \pi(M_1)$. Since the linear span of $u_g L^\infty(X_2,\mu_2)$, $g \in \cG_2$, is dense in $L^2(N_2,\om_2)$, we certainly have the density of the linear span of $u_g M_2$, $g \in \cG_2$.

Since $\vphi_1 \circ \pi^{-1} \circ E \prec \om_2$, since $\om_2 = \vphi_2 \circ E$ and since $M_1^{\vphi_1}$ is a factor, it follows from Lemma \ref{lem.class} that $\class(\vphi_1) \prec \class(\om_2)$. We have proven above that $\class(\vphi_1) = \stc(\al_1)$ and $\class(\om_2) = \stc(\al_2)$. So the theorem is proven.
\end{proof}

Theorem \ref{thm.main} is an immediate consequence of Theorem \ref{thm.main3}, as we show now.

\begin{proof}[{Proof of Theorem \ref{thm.main}}]
Given equivalent probability measures $\nu \sim \eta$ on a standard Borel space $Y$ and given a nonamenable group $\Lambda$, the nonsingular Bernoulli action of $\Gamma = \Z \ast \Lambda$ given by \eqref{eq.my-bernoulli} is isomorphic with the action defined by \eqref{eq.more-general-nonsingular} associated with $G = \Z \actson^\al (Z,\zeta) = \prod_{n \in \Z} (Y,\mu_n)$, where $\mu_n = \nu$ if $n \in \N$ and $\mu_n = \eta$ if $n \in \Z \setminus \N$. Define $\gamma = \stc((\log d\nu / d\eta)_*(\nu))$.

By Theorem \ref{thm.main3}, it suffices to prove that $\gamma \sim \stc(\al)$. By definition, $\stc(\al)$ is the smallest stable measure class satisfying $(\log d(n \cdot \zeta)/d\zeta)_*(\zeta) \prec \stc(\al)$ for all $n \in \Z$. The measure $(\log d(n \cdot \zeta)/d\zeta)_*(\zeta)$ is equivalent with the $|n|$-fold convolution power of $(\log d\nu / d\eta)_*(\nu)$ or its opposite, depending on the sign of $n$. So, $\stc(\al) \sim \stc((\log d\nu / d\eta)_*(\nu)) = \gamma$.
\end{proof}

To prove Theorem \ref{thm.main2}, we use the following lemma about the relation between independent Borel sets and convolution products. This result is essentially contained in \cite[Section 3]{LP97}, but we provide an elementary proof for convenience.

As mentioned in the introduction, recall that a Borel set $K \subset \R$ is said to be independent if every finite subset $\cF \subset K$ generates a free abelian subgroup of $\R$ of rank $|\cF|$. In other words, a Borel set $K \subset \R$ is independent if and only if for every $n$-tuple of distinct elements $x_1,\ldots,x_n \in K$, the homomorphism $\Z^n \to \R : \lambda \mapsto \sum_{k=1}^n \lambda_k x_k$ is injective.

Also recall that given a measure class $\mu$ on $\R$, we define $\mutil$ by $\mutil(\cU) = \mu(-\cU)$ and we denote by $\stc(\mu)$ the join of the measure classes $\mu^{\ast n} \ast {\mutil}^{\ast m}$, $n,m \geq 0$.

\begin{lemma}\label{lem.indep-K}
Let $K \subset \R$ be an independent Borel set. We decompose any probability measure $\eta$ on $\R$ as the sum $\eta = \eta_a + \eta_c$ of an atomic and a nonatomic measure.
\begin{enumlist}
\item If $\eta$ is a nonatomic probability measure on $\R$, the set $C = \{x \in \R \mid \eta(x+K) > 0\}$ is countable. For any probability measure $\eta$ on $\R$, we define the measure class $\pi_K(\eta)$ on $K$ by $\pi_K(\eta) := \bigvee_{x \in C} (\delta_{-x} \ast \eta_c)|_{K}$.
\item Let $\mu$ be a nonatomic probability measure on $\R$ and let $\eta$ be any probability measure on $\R$. Denote by $\eta_a$ the atomic part of $\eta$. Then, $\pi_K(\eta \ast \mu) \sim \pi_K(\eta_a \ast \mu)$. In particular, if also $\eta$ is nonatomic, then $\pi_K(\eta \ast \mu) = 0$. If $\eta_a \neq 0$, we conclude that $\pi_K(\eta \ast \mu) \sim \pi_K(\mu)$.
\item For every probability measure $\mu$ on $\R$ and every atomic probability measure $\rho$ on $\R$, we have that $\pi_K(\rho \ast \stc(\mu)) \sim \pi_K(\mu_c) \vee \pi_K(\mutil_c)$.
\item If $x \in \R$ and $\mu$ is a probability measure with $\mu(\R \setminus K) = 0$, then $\pi_K(\delta_x \ast \mu) \sim \mu_c$ and $\pi_K(\delta_x \ast \mutil) = 0$.
\end{enumlist}
\end{lemma}
\begin{proof}
We first prove the following two claims.
\begin{enumlist}[label=(\roman*),labelwidth=4ex,leftmargin=4.5ex]
\item If $x \in \R \setminus \{0\}$, then $(x+K) \cap K$ has at most one element.
\item If $x \in \R$, then $(x-K) \cap K$ has at most two elements
\end{enumlist}
To prove (i), if $(x+K) \cap K$ is nonempty and $x \neq 0$, we can write $x = a-b$ with $a,b \in K$ and $a \neq b$. Since $K$ is independent, one deduces that $(x+K) \cap K = \{a\}$. To prove (ii), if $(x-K) \cap K$ is nonempty, we can write $x = a+b$ with $a,b \in K$. Since $K$ is independent, one deduces that $(x-K) \cap K = \{a,b\}$.

1.\ Fix a nonatomic probability measure $\eta$ on $\R$. If $x \neq y$, it follows from (i) that $\eta((x+K) \cap (y+K)) = 0$. Since $\eta$ is a finite measure, the set $C = \{x \in \R \mid \eta(x+K) > 0\}$ is countable.

2.\ Fix $x \in \R$. Then, $(\eta \ast \mu)(x+K) = \int_\R \mu(-y+x + K) \, d\eta(y)$. There are only countable many $y \in \R$ such that $\mu(-y+x + K) > 0$. Thus,
$$(\eta_c \ast \mu)(x+K) = \int_\R \mu(-y+x + K) \, d\eta_c(y) = 0 \; .$$
This holds for every $x \in \R$, so that $\pi_K(\eta_c \ast \mu) = 0$. It follows that $\pi_K(\eta \ast \mu) \sim \pi_K(\eta_a \ast \mu)$. By definition, $\pi_K(\delta_x \ast \mu) \sim \pi_K(\mu)$ for every $x \in \R$. Then also $\pi_K(\eta_a \ast \mu) \sim \pi_K(\mu)$, whenever $\eta_a \neq 0$.

3.\ Write $\rho_1 = \rho \ast \stc(\mu_a)$. Then, $\rho \ast \stc(\mu) \sim \rho_1 \ast \stc(\mu_c)$. By 2, we know that $\pi_K(\rho_1 \ast \mu_c^{\ast n} \ast \mutil_c^{\ast m}) = 0$ when $n + m \geq 2$. Since $\rho_1$ is atomic, also $\pi_K(\rho_1 \ast \delta_0) = 0$. Therefore,
$$\pi_K(\rho \ast \stc(\mu)) \sim \pi_K(\rho_1 \ast \stc(\mu_c)) \sim \pi_K(\rho_1 \ast \mu_c) \vee \pi_K(\rho_1 \ast \mutil_c) \sim \pi_K(\mu_c) \vee \pi_K(\mutil_c) \; .$$

4.\ By definition, $\pi_K(\delta_x \ast \mu) \sim \pi_K(\delta_x \ast \mu_c)$. When $y \neq x$, it follows from (i) that $(\delta_x \ast \mu_c)(y + K) = 0$. Also, $\delta_x \ast \mu_c$ is supported on $x + K$. Thus, $\pi_K(\delta_x \ast \mu) \sim \delta_{-x} \ast \delta_x \ast \mu_c = \mu_c$.

We also have $\pi_K(\delta_x \ast \mutil) = \pi_K(\delta_x \ast \mutil_c)$. By (ii), for every $y \in \R$, we have
$$(\delta_x \ast \mutil_c)(y + K) = \mutil_c(-x + y + K) = \mu_c(x-y-K) = \mu_c(K \cap (x-y-K)) = 0 \; .$$
This holds for all $y \in \R$ and thus, $\pi_K(\delta_x \ast \mutil) = 0$.
\end{proof}

We can then deduce the following consequence of Theorem \ref{thm.main3}.

\begin{corollary}\label{cor.Bernoulli-independent-set}
For $i \in \{1,2\}$, let $\nu_i \sim \eta_i$ be equivalent probability measures on a standard Borel space $Y_i$. Let $\Lambda_i$ be nonamenable groups. Denote $\Gamma_i = \Z \ast \Lambda_i$ and consider the nonsingular Bernoulli action $\Gamma_i \actson (X_i,\mu_i)$ given by \eqref{eq.my-bernoulli}. Write $\sigma_i = (\log d\nu_i / d\eta_i)_*(\nu_i)$. Denote $M_i = L^\infty(X_i) \rtimes \Gamma_i$. Let $K \subset \R$ be an independent Borel set and use the notation $\pi_K$ introduced in Lemma \ref{lem.indep-K}.
\begin{enumlist}
\item If $M_1 \cong M_2$, then $\pi_K(\sigma_1) \vee \pi_K(\sigmatil_1) \sim \pi_K(\sigma_2) \vee \pi_K(\sigmatil_2)$.
\item If $M_1$ admits an embedding with expectation into $M_2$, then $\pi_K(\sigma_1) \vee \pi_K(\sigmatil_1) \prec \pi_K(\sigma_2) \vee \pi_K(\sigmatil_2)$.
\end{enumlist}
\end{corollary}
\begin{proof}
Considering $G_i = \Z \actson^{\al_i} (Z_i,\zeta_i) = \prod_{n \in \Z} (Y_i,\mu_{i,n})$, where $\mu_{i,n} = \nu_i$ if $n \in \N$ and $\mu_{i,n} = \eta_i$ if $n \in \Z \setminus \N$, we have seen in the proof of Theorem \ref{thm.main} that $\stc(\al_i) \sim \stc(\sigma_i)$.

If $M_1$ admits an embedding with expectation into $M_2$, Theorem \ref{thm.main3} provides an atomic probability measure $\rho$ such that $\stc(\al_1) \prec \rho \ast \stc(\al_2)$. Thus, $\stc(\sigma_1) \prec \rho \ast \stc(\sigma_2)$. Applying $\pi_K$ and using Lemma \ref{lem.indep-K}.3, we conclude that $\pi_K(\sigma_1) \vee \pi_K(\sigmatil_1) \prec \pi_K(\sigma_2) \vee \pi_K(\sigmatil_2)$. If $M_1 \cong M_2$, also the converse absolute continuity holds.
\end{proof}

\begin{proof}[{Proof of Theorem \ref{thm.main2}}]
In the context of Theorem \ref{thm.main2}, the measure $\sigma_i = (\log d\nu_i / d\eta_i)_*(\nu_i)$ is a translate of $\nu_i$ and $\nu_i$ is supported on $K$. By Lemma \ref{lem.indep-K}.4, we get that $\pi_K(\sigma_i) \sim \nu_i$ and $\pi_K(\sigmatil_i) = 0$. The result thus follows from Corollary \ref{cor.Bernoulli-independent-set}.
\end{proof}

\begin{example}\label{exam.full-with-standard-tau}
Fix a compact independent set $K \subset \R$ such that $K$ is homeomorphic to a Cantor set (see e.g.\ \cite[Theorems 5.1.4 and 5.2.2]{Rud62}). Fix a countable nonamenable group $\Lambda$ and put $\Gamma = \Z \ast \Lambda$. Put $Y = [0,1] \cup K$.

Given any nonatomic probability measure $\rho$ on $K$, define the probability measure $\nu_\rho$ on $Y$ as $(\lambda + \rho)/2$, where $\lambda$ is the Lebesgue measure on $[0,1]$. Define the probability measure $\eta_\rho$ on $Y$ by normalizing $\exp(-y) \, d\nu_\rho(y)$. Consider the associated nonsingular Bernoulli action $\Gamma \actson^{\al_\rho} (X,\mu_\rho)$ given by \eqref{eq.my-bernoulli}, with crossed product factor $M_\rho = L^\infty(X,\mu_\rho) \rtimes_{\al_\rho} \Gamma$.
\begin{enumlist}
\item $M_\rho$ is a full factor of type III and Connes $\tau$-invariant of $M_\rho$ is the usual topology on $\R$. When $\Lambda$ has infinite conjugacy classes and is not inner amenable, this was proven in \cite[Proposition 7.1]{VW17}, but the result holds by only assuming that $\Lambda$ is nonamenable, as we show in Lemma \ref{lem.full} below.
\item Let $\rho$ and $\rho'$ be nonatomic probability measures on $K$. If $M_{\rho} \cong M_{\rho'}$, then $\rho \sim \rho'$. If $M_{\rho}$ admits an embedding with expectation into $M_{\rho'}$, then $\rho \prec \rho'$. Both statements follow from Corollary \ref{cor.Bernoulli-independent-set}: given a nonatomic probability measure $\rho$ on $K$, we have that $(\log d\nu_\rho / d\eta_\rho)_*(\nu_\rho)$ is a translate of $\nu_\rho$. Since $\rho$ is supported on $K$ and since $\lambda(x+ K) = 0$ for every $x \in \R$, it follows from Lemma \ref{lem.indep-K} that $\pi_K(\nu_\rho) \vee \pi_K(\widetilde{\nu_\rho}) \sim \rho$.
\end{enumlist}
\end{example}

For completeness, we include a proof for the following result, which was proven in \cite[Proposition 7.1]{VW17} under the stronger assumption that $\Lambda$ has infinite conjugacy classes and is not inner amenable.

\begin{lemma}\label{lem.full}
Let $\nu \sim \eta$ be equivalent probability measures on a standard Borel space $Y$. Assume that $\nu$ and $\eta$ are not concentrated on a single atom. Let $\Lambda$ be a countable nonamenable group. Write $\Gamma = \Z \ast \Lambda$ and define the nonsingular Bernoulli action $\Gamma \actson (X,\mu)$ by \eqref{eq.my-bernoulli}. Then, the factor $M = L^\infty(X,\mu) \rtimes \Gamma$ is full and the $\tau$-invariant is the weakest topology on $\R$ that makes the map
$$\R \to \cU(L^\infty(Y,\nu)) : t \mapsto \Bigl(\frac{d\nu}{d\eta}\Bigr)^{it}$$
continuous, where $\cU(L^\infty(Y,\nu))$ is equipped with the strong topology.
\end{lemma}
\begin{proof}
Denote by $\vphi$ the canonical crossed product state on $M$. Write $Q = L^\infty(X,\mu) \rtimes \Lambda$. As in the proof of \cite[Proposition 7.1]{VW17}, it suffices to show the following: if $x_n \in \cU(M)$ is a sequence of unitaries such that $x_n a - a x_n \to 0$ $*$-strongly for every $a \in Q$, then $x_n - \vphi(x_n) 1 \to 0$ $*$-strongly. Note that $Q \subset M^\vphi$. There thus exists a unique $\vphi$-preserving conditional expectation $E : M \to L(\Lambda)$. In the proof of \cite[Proposition 7.1]{VW17}, it is shown that $x_n - E(x_n) \to 0$ $*$-strongly. Writing $y_n = E(x_n)$, we have found a bounded sequence in $L(\Lambda)$ satisfying $y_n a - a y_n \to 0$ $*$-strongly for every $a \in Q$.

For every $h \in \Lambda$, denote by $\pi_h : L^\infty(Y,\eta) \to L^\infty(X,\mu)$ the state preserving embedding as the $h$-th coordinate. Choose an element $a \in L^\infty(Y,\eta)$ such that $\int_Y |a|^2 \, d\eta = 1$ and $\int_Y a \, d\eta = 0$. Denote by $(u_g)_{g \in \Lambda}$ the canonical unitaries, so that $(y_n)_g := \vphi(y_n u_g^*)$ are the canonical Fourier coefficients. We also write $\|x\|_{\vphi} = \vphi(x^* x)^{1/2}$ for all $x \in M$. Since $\pi_e(a)$ and $\pi_h(a)$ are orthogonal in $L^2(M,\vphi)$ for all $h \neq e$, a direct computation shows that
$$\|y_n \pi_e(a) - \pi_e(a) y_n\|_\vphi^2 \geq 2 \sum_{h \in \Lambda \setminus \{e\}} |(y_n)_h|^2 = 2 \, \|y_n - \vphi(y_n)1\|_2^2 \; .$$
Therefore, $y_n - \vphi(y_n)1 \to 0$ $*$-strongly. Then also $x_n - \vphi(x_n)1 \to 0$ $*$-strongly.
\end{proof}

\section{Solidity of nonsingular Bernoulli actions, proof of Theorem~\ref{thm.main-solid}}

Recall from \cite{Oza03,VV05} that a von Neumann algebra $M$ is called \emph{solid} if for every diffuse von Neumann subalgebra with expectation $A \subset M$, the relative commutant $A' \cap M$ is amenable. If $B \subset M$ is a von Neumann subalgebras with expectation, recall from \cite{Mar16} that $M$ is said to be \emph{solid relative to $B$} if the following holds: for every nonzero projection $p \in M$ and nonamenable von Neumann subalgebra $Q \subset pMp$ with expectation and with diffuse center $\cZ(Q)$, we have that $Q \prec_M B$.

Recall from \cite{CI08} that a countable nonsingular equivalence relation $\cR$ on a standard probability space $(X,\mu)$ is called \emph{solid} if for every Borel subequivalence relation $\cS \subset \cR$, there exists a partition of $X$, up to measure zero, into $\cS$-invariant Borel subsets $(X_n)_{n \geq 0}$ such that $\cS|_{X_0}$ is amenable and $\cS|_{X_n}$ is ergodic for all $n \geq 1$. This is equivalent to saying that for every diffuse von Neumann subalgebra $A \subset L^\infty(X)$, the relative commutant $A' \cap L(\cR)$ is amenable. Finally, an essentially free nonsingular action $\Gamma \actson (X,\mu)$ is said to be a \emph{solid action} if the orbit equivalence relation $\cR(\Gamma \actson X)$ is solid.

\begin{definition}\label{def.solid-state}
A faithful normal state $\vphi$ on a von Neumann algebra $M$ is said to be a \emph{solid state} if for every faithful normal state $\psi$ on $M$ with nonamenable centralizer $M^\psi$, we have that $\psi \prec \vphi$.
\end{definition}

Theorem \ref{thm.main-solid} is a special case of the following result.

\begin{theorem} \label{thm.solid-more-general}
Let $G$ be a countable amenable group and let $G \actson (Z,\zeta)$ be a nonsingular action on a nontrivial standard probability space. Let $\Lambda$ be a countable nonamenable group. Define $\Gamma = G \ast \Lambda$ and let $\Gamma \actson (X,\mu)$ be defined by \eqref{eq.more-general-nonsingular}. Put $M = L^\infty(X,\mu) \rtimes \Gamma$.
\begin{enumlist}
\item $\Gamma \actson (X,\mu)$ is a solid action.
\item The factor $M$ is solid relative to $L(\Lambda)$.
\item If $\Lambda$ is biexact, then $M$ is solid.
\item If $\Lambda$ is biexact and $(\log d(g \cdot \zeta)/d\zeta)_*(\zeta)$ is nonatomic for every $g \in G \setminus \{e\}$, then the crossed product state $\vphi_\mu$ on $M$ is a solid state.
\end{enumlist}
\end{theorem}

We first prove the following lemma, in which we also introduce some of the notation that will be used in the proof of Theorem \ref{thm.solid-more-general}.

\begin{lemma}\label{lem.solid-more-general}
Make the same assumptions as in Theorem \ref{thm.solid-more-general}. Write $I = (G \ast \Lambda)/G$. Denote by $\vphi$ the crossed product state on $M$ induced by $\mu$. For every $a \in G^{(I)}$, denote by $\mu_a \sim \mu$ the probability measure on $X$ defined by
$\mu_a = \prod_{i \in I} a_i \cdot \zeta$. Denote by $\vphi_a$ the corresponding crossed product state on $M$.

Let $p \in \core(M)$ be a projection of finite trace. Let $A \subset p \core(M) p$ be a von Neumann subalgebra whose relative commutant $A' \cap p \core(M) p$ has no amenable direct summand. Then one of the following statements holds.
\begin{enumlist}
\item $A \prec_{\core(M)} L(\Lambda) \vee L_\vphi(\R)$.
\item $A \prec_{\core(M)} L_{\vphi_a}(\R)$ for some $a \in G^{(I)}$.
\end{enumlist}
If moreover $\Lambda$ is biexact, then the second statement always holds.
\end{lemma}
\begin{proof}
Denote $P = L^\infty(Z,\zeta) \rtimes G$ and let $\om$ be the crossed product state on $P$ given by $\zeta$. Write $(N,\om) = (P,\om)^I$ and $\cN = N \rtimes \Gamma$. We still denote by $\om$ the crossed product state on $\cN$. Denote by $\theta : M \to \cN$ the embedding given by Lemma \ref{lem.canonical-embedding}. Note that there is a unique faithful normal conditional expectation $E : \cN \to \theta(M)$ such that $\vphi \circ \theta^{-1} \circ E = \om$. For the rest of the proof, we view $M$ as a subalgebra of $\cN$ and no longer write the canonical embedding $\theta$. We then also replace the notation $\om$ by $\vphi$.

The conditional expectation $E$ induces an embedding $\core(M) \subset \core(\cN)$. Write $B = L_\vphi(\R)$ and $Q = \cN_{p \core(\cN) p}(A)\dpr$. Since $A' \cap p \core(\cN) p \supset A' \cap p \core(M) p$ has no amenable direct summand, by Theorem \ref{thm.def-rig-nc-bernoulli}, we can take a projection $z \in \cZ(Q)$ such that $Az \prec_{f,\core(\cN)} B$ and, with $z' = p - z$, we have $Q z' \prec_{f,\core(\cN)} B \vee L(\Gamma)$.

Assume that $z' \neq 0$. We prove that $A z' \prec_{\core(\cN)} B \vee L(\Lambda)$. Assume the contrary. In particular, $Q z' \not\prec_{\core(\cN)} B \vee L(\Lambda)$. We may view $B \vee L(\Gamma)$ as the tensor product $L(G \ast \Lambda) \ovt B$ and hence, also as the amalgamated free product $(L(G) \ovt B) \ast_B (L(\Lambda) \ovt B)$. Since $Q$ contains the commuting subalgebras $A$ and $A' \cap p \core(\cN) p$, and since $A' \cap p \core(\cN) p$ has no amenable direct summand, our assumption that $A z' \not\prec_{\core(\cN)} B \vee L(\Lambda)$ and \cite[Theorem 4.2]{CH08} imply that $A z' \prec_{\core(\cN)} B \vee L(G)$. Our assumption says in particular that $A z' \not\prec_{\core(\cN)} B$, so that \cite[Theorem 2.4]{CH08} implies that $A' \cap p \core(\cN) p \prec_{\core(\cN)} B \vee L(G)$. This is a contradiction because $B \vee L(G)$ is amenable, while $A' \cap p \core(\cN) p$ has no amenable direct summand.

Since $z$ and $z'$ cannot be both equal to zero, we have proven that $A \prec_{\core(\cN)} L(\Lambda) \vee L_\vphi(\R)$. To conclude the proof of the lemma, we assume that none of the two statements in the lemma hold and prove that $A \not\prec_{\core(\cN)} L(\Lambda) \vee L_\vphi(\R)$. Since the two statements in the lemma do not hold, we can choose a sequence of unitaries $w_n \in \cU(A)$ such that
\begin{equation}\label{eq.we-have-this}
\|E_{L(\Lambda) \vee L_\vphi(\R)}(x^* w_n y)\|_{2,\Tr} \to 0 \quad\text{and}\quad \|E_{L_{\vphi_a}(\R)}(x^* w_n y)\|_{2,\Tr} \to 0
\end{equation}
for all $a \in G^{(I)}$ and $x,y \in p \core(M)$. Here, all conditional expectations are the unique trace preserving ones. It suffices to prove that
\begin{equation}\label{eq.goal-here}
\|E_{L(\Lambda) \vee L_\vphi(\R)}(x^* w_n y)\|_{2,\Tr} \to 0 \quad\text{for all $x,y \in p \core(\cN)$.}
\end{equation}
As in the proof of Theorem \ref{thm.main3}, we denote by $\Gamma \actson^\al G^{(I)}$ the natural action by translation, define the generalized wreath product group $\cG = G \wr_I \Gamma = G^{(I)} \rtimes_\al \Gamma$ and view $\cN$ as the crossed product $L^\infty(X) \rtimes \cG$. Define $i_0 \in I$ as the coset $i_0 = G$. Under this identification, $M = L^\infty(X) \rtimes \theta(\Gamma)$, where $\theta : \Gamma \to \cG$ is the injective group homomorphism uniquely determined by $\theta(g) = \pi_{i_0}(g) \, g$ and $\theta(\lambda) = \lambda$ for all $g \in G$, $\lambda \in \Lambda$. In particular, every $a \in G^{(I)}$ gives rise to a canonical unitary $u_a \in N$. We have $\vphi_a \circ E = \vphi \circ \Ad u_a^*$, so that $L_{\vphi_a}(\R) = u_a \, L_\vphi(\R) \, u_a^*$. By density, it suffices to prove \eqref{eq.goal-here} for $x = x_0 u_a$ and $y = y_0 u_b$ with $x_0, y_0 \in p \core(M)$ and $a,b \in G^{(I)}$.

If $a=b=e$, then \eqref{eq.goal-here} follows immediately from \eqref{eq.we-have-this}. When $a$ and $b$ are not both equal to $e$, the set $\cF = \{\lambda \in \Lambda \mid \al_\lambda(b) = a\}$ is finite. Also, for $g \in \Gamma$, we have that $a^{-1} \theta(g) b \in \theta(\Lambda) = \Lambda$ if and only if $g \in \cF$. So, if $\cF = \emptyset$, we find that
$$E_{L(\Lambda) \vee L_\vphi(\R)}(u_a^* x_0^* w_n y_0 u_b) = 0$$
for all $n \in \N$. When $\cF \neq \emptyset$, a direct computation shows that for every $x_1 \in \core(M)$,
$$E_{L(\Lambda) \vee L_\vphi(\R)}(u_a^* x_1 u_b) = \sum_{\lambda \in \cF} u_a^* \, E_{L_{\vphi_a}(\R)}(x_1 \, u_\lambda^*) \, u_a \, u_\lambda \; .$$
Since for every $\lambda \in \cF$, we have by \eqref{eq.we-have-this} that $\|E_{L_{\vphi_a}(\R)}(x_0^* w_n y_0 u_\lambda^*)\|_{2,\Tr} \to 0$, again \eqref{eq.goal-here} follows. So the first part of the lemma is proven.

Finally assume that $\Lambda$ is moreover biexact. Above, we have seen that $Q z' \prec_{\core(\cN)} B \vee L(\Lambda)$. We can view $B \vee L(\Lambda)$ as the tensor product $L(\Lambda) \ovt B$, where $B$ is abelian. By \cite[Propositions 11 and 12]{OP03}, any von Neumann subalgebra $D$ of a corner of $L(\Lambda) \ovt B$ having a nonamenable relative commutant, intertwines into $B$. So we find that $A z' \prec_{\core(\cN)} B$. Since $z$ and $z'$ cannot be both equal to zero, we get that $A \prec_{\core(\cN)} L_\vphi(\R)$. The argument above then shows that $A \prec_{\core(M)} L_{\vphi_a}(\R)$ for some $a \in G^{(I)}$.
\end{proof}

\begin{proof}[{Proof of Theorem \ref{thm.solid-more-general}}]
We start by proving that $M$ is solid relative to $L(\Lambda)$. It suffices to prove the following statement: if $e \in M$ is a projection and $A \subset eMe$ is a diffuse abelian von Neumann subalgebra with expectation such that the relative commutant $Q = A' \cap eMe$ has no amenable direct summand, then $Q \prec_M L(\Lambda)$. Fix a faithful normal conditional expectation $E : eMe \to A$ and choose a faithful normal state $\psi$ on $A$. We still denote by $\psi$ the state $\psi \circ E$ on $eMe$. We identify $\core(eMe) = e \core(M) e$. Fix a nonzero projection $p \in L_\psi(\R)$ of finite trace. Then, $A p$ and $p \core_\psi(Q) p$ are commuting von Neumann subalgebras of $p \core(M) p$ and $p \core_\psi(Q) p$ has no amenable direct summand.

By Lemma \ref{lem.solid-more-general} and using the notation introduced in that lemma, one of the following statements holds.
\begin{itemlist}
\item $A p \prec_{\core(M)} L(\Lambda) \vee L_\vphi(\R)$.
\item $A p \prec_{\core(M)} L_{\vphi_a}(\R)$ for some $a \in G^{(I)}$.
\end{itemlist}
We claim that the second statement does not hold. Since $A$ is diffuse, we can choose a sequence $w_n \in \cU(A)$ such that $w_n \to 0$ weakly. Whenever $x_0,y_0 \in M$ and $x_1,y_1 \in L_{\vphi_a}(\R)$, we get that
$$E_{L_{\vphi_a}(\R)}(x_1^* x_0^* \, w_n \, y_0 y_1) = x_1^* \, \vphi_a(x_0^* \, w_n  y_0) \, y_1 \; .$$
Since $w_n \to 0$ weakly, it follows by density that $\|E_{L_{\vphi_a}(\R)}(x^* \, w_n \, y)\|_{2,\Tr} \to 0$ for all $x,y \in \core(M)$ with $\Tr(x^*x) < +\infty$ and $\Tr(y^* y) < +\infty$. In particular, $\|E_{L_{\vphi_a}(\R)}(x^* \, w_n p \, y)\|_{2,\Tr} \to 0$ for all $x,y \in \core(M)$. So, the claim is proven. It follows that $A p \prec_{\core(M)} L(\Lambda) \vee L_\vphi(\R)$.

We now claim that $A \prec_M L(\Lambda)$. Assume the contrary. Denote by $E_{L(\Lambda)} : M \to L(\Lambda)$ the unique $\vphi$-preserving conditional expectation. Since $A$ is abelian and $A \not\prec_M L(\Lambda)$, we can take a sequence of unitaries $w_n \in \cU(A)$ such that $E_{L(\Lambda)}(x^* w_n y) \to 0$ $*$-strongly, for all $x,y \in M$. If now $x_0,y_0 \in M$ and $x_1,y_1 \in L_\vphi(\R)$, we get that
$$E_{L(\Lambda) \vee L_\vphi(\R)}(x_1^* x_0^* \, w_n \, y_0 y_1) = x_1^* \, E_{L(\Lambda)}(x_0^* \, w_n \, y_0) \, y_1 \; .$$
By density, we get that $\|E_{L(\Lambda) \vee L_\vphi(\R)}(x^* \, w_n \, y)\|_{2,\Tr}$ for all $x,y \in \core(M)$ with $\Tr(x^*x) < +\infty$ and $\Tr(y^* y) < +\infty$. In particular, $\|E_{L(\Lambda) \vee L_\vphi(\R)}(x^* \, w_n p \, y)\|_{2,\Tr} \to 0$ for all $x,y \in \core(M)$. This contradicts the statement that $A p \prec_{\core(M)} L(\Lambda) \vee L_\vphi(\R)$. So, the claim that $A \prec_M L(\Lambda)$ is proven.

Choose projections $r \in A$ and $s \in L(\Lambda)$, a nonzero partial isometry $v \in r M s$ and a unital normal $*$-homomorphism $\theta : r A r \to s L(\Lambda) s$ such that $a v = v \theta(a)$ for all $a \in r A r$. Denote $D = \theta(r A r)' \cap s M s$. Let $\Theta : M \to P^I \rtimes \Gamma$ be the embedding given by Lemma \ref{lem.canonical-embedding}. Then, $\Theta(\theta(rAr))$ is a diffuse von Neumann subalgebra of a corner of $L(\Lambda)$. Since $\Lambda \cap \Stab i = \{e\}$ for every $i \in I$ and $\Theta(\theta(rAr))$ is diffuse, we get for every $i \in I$ that $\Theta(\theta(rAr)) \not\prec_{L(\Gamma)} L(\Stab i)$. It then follows from Proposition \ref{prop.control-normalizer} that $\Theta(D) \subset L(\Gamma)$. Since $\Theta(M) \cap L(\Gamma) = L(\Lambda)$, we conclude that $D \subset s L(\Lambda) s$. In particular $s_1 = v^* v$ belongs to $L(\Lambda)$. By construction, $v^* Q v \subset D$ and $r_1 = vv^*$ belongs to $Q$. In particular, $Q \prec_M L(\Lambda)$. So we have proven that $M$ is solid relative to $L(\Lambda)$.

If $\Lambda$ is biexact, then $L(\Lambda)$ is solid by \cite{Oza03}. Since $M$ is solid relative to $L(\Lambda)$, it then follows that $M$ is solid, when $\Lambda$ is biexact.

We next prove that $\Gamma \actson (X,\mu)$ is a solid action. Choose a diffuse von Neumann subalgebra $A \subset L^\infty(X)$. We have to prove that $A' \cap M$ is amenable. Assume that $A' \cap M$ is nonamenable. Since $L^\infty(X) \subset M$ is an inclusion with expectation, also $A \subset M$ and $A' \cap M$ are inclusions with expectation. Since $A \subset \cZ(A' \cap M)$ and $M$ is solid relative to $L(\Lambda)$, we find that $A' \cap M \prec_M L(\Lambda)$. A fortiori, $A \prec_M L(\Lambda)$. On the other hand, since $A$ is diffuse, we can take a sequence of unitaries $w_n \in \cU(A)$ such that $w_n \to 0$ weakly. For all $g,h \in \Gamma$ and $x,y \in L^\infty(X)$, we have
$$E_{L(\Lambda)}(u_g^* x^* \, w_n \, y u_h) = \begin{cases} \vphi(x^* \, w_n \, y) \; u_g^* u_h &\;\;\text{if $g^{-1} h \in \Lambda$,}\\ 0 &\;\;\text{otherwise.}\end{cases}$$
The conditional expectation $E_{L(\Lambda)}$ is $\vphi$-preserving and the restriction of $\vphi$ to $L(\Lambda)$ is the canonical trace on $L(\Lambda)$. By density, we find that $\|E_{L(\Lambda)}(x^* \, w_n \, y)\|_{2,\vphi} \to 0$ for all $x, y \in M$. So, $A \not\prec_M L(\Lambda)$. This contradiction concludes the proof that $\Gamma \actson (X,\mu)$ is a solid action.

Finally assume that $\Lambda$ is biexact and that $(\log d(g \cdot \zeta)/d\zeta)_*(\zeta)$ is nonatomic for every $g \in G \setminus \{e\}$. Let $\psi$ be a faithful normal state on $M$ such that $M^\psi$ is nonamenable. Take a nonzero central projection $e \in \cZ(M^\psi)$ such that $M^\psi e$ has no amenable direct summand. Fix a finite trace projection $q \in L_\psi(\R) \subset \core(M)$ such that the projection $p = e q$ is nonzero. Then, $L_\psi(\R) p$ is a von Neumann subalgebra of $p \core(M) p$ whose relative commutant contains $M^\psi p$ and hence, has no amenable direct summand. By the second part of Lemma \ref{lem.solid-more-general}, we find $a \in G^{(I)}$ such that $L_\psi(\R) p \prec_{\core(M)} L_{\vphi_a}(\R)$. By Lemma \ref{lem.HSV}, it follows that $\psi \prec \vphi_a$. In particular, $\vphi_a$ has a nonamenable centralizer.

Denote by $\Gamma \actson^\al G^{(I)}$ the action by translation. Let $i_0 \in I$ be the coset $G$. Denote by $\pi_i : G \to G^{(I)}$ the embedding in the $i$'th coordinate. A map $c : \Gamma \to G^{(I)}$ is called an $\al$-cocycle if $c(gh) = c(g) \, \al_g(c(h))$ for all $g, h \in \Gamma$. Let $c : \Gamma \to G^{(I)}$ be the unique $\al$-cocycle satisfying $c(g) = \pi_{i_0}(g)$ for all $g \in G$ and $c(\lambda) = e$ for all $\lambda \in \Lambda$. A direct computation gives that $g \cdot \mu_a = \mu_{c(g) \, \al_g(a)}$ for all $g \in \Gamma$. Define the subgroup $L \subset \Gamma$ by
$$L = \{g \in \Gamma \mid c(g) = a \, \al_g(a^{-1}) \} \; .$$
Since $(\log d(g \cdot \zeta)/d\zeta)_*(\zeta)$ is nonatomic for every $g \in G \setminus \{e\}$, we get that $(d \mu_b / d \mu_c)(x) \neq 1$ for a.e.\ $x \in X$ and all $b \neq c$ in $G^{(I)}$. It follows that $(d(g \cdot \mu_a) / d\mu_a)(x) \neq 1$ for a.e.\ $x \in X$ and all $g \in \Gamma \setminus L$. It follows that $M^{\vphi_a}$ is a von Neumann subalgebra with expectation of $L^\infty(X) \rtimes L$. Since $M^{\vphi_a}$ is nonamenable, we conclude that $L$ is a nonamenable group.

For every $b \in G^{(I)}$, we denote $|b| = \# \{i \in I \mid b_i \neq e\}$. We also define for every $g \in \Gamma = G \ast \Lambda$, the $G$-length $|g|_G$ as the minimal number of elements in $G$ one needs when writing $g$ as a product of elements in $G$ and elements in $\Lambda$. Thus, $|\lambda|_G = 0$ for all $\lambda \in \Lambda$ and $|g|_G = n$ whenever $g = \lambda_0 g_1 \lambda_1 \cdots \lambda_{n-1} g_n \lambda_n$ with $g_i \in G \setminus \{e\}$ for all $i$, $\lambda_i \in \Lambda \setminus \{e\}$ for all $i \in \{1,\ldots,n-1\}$ and $\lambda_0,\lambda_n \in \Lambda$. Another direct computation shows that $|c(g)| = |g|_G$ for all $g \in \Gamma$.

For all $g \in L$, we have that $|g|_G = |c(g)| = |a \, \al_g(a^{-1})| \leq 2 \, |a|$. Hence, $g \mapsto |g|_G$ is bounded on $L$. Since $L$ is nonamenable, this implies that $L = g_0 \Lambda_0 g_0^{-1}$ for some $g_0 \in \Gamma$, where $\Lambda_0 \subset \Lambda$ is a nonamenable subgroup. Write $b = c(g_0)$. Since $c(\lambda) = e$ for all $\lambda \in \Lambda$, we get that $c(g) = b \, \al_g(b^{-1})$ for all $g \in g_0 \Lambda g_0^{-1}$. Hence, $a \, \al_g(a^{-1}) = b \, \al_g(b^{-1})$ for all $g \in L$. This means that $\al_g(b^{-1} a) = b^{-1} a$ for all $g \in L$. Since the nonamenable group $L$ acts with infinite orbits on $I$, we conclude that $a = b$. So, $\mu_a = \mu_{c(g_0)} = g_0 \cdot \mu$. Using the unitary $u_{g_0} \in M$, it follows that $\vphi_a$ and $\vphi$ are unitarily conjugate. We already proved that $\psi \prec \vphi_a$. It follows that $\psi \prec \vphi$. We have thus proven that $\vphi$ is a solid state on $M$.
\end{proof}

%

Having proven Theorem \ref{thm.main-solid}, it is tempting to believe that for any nonsingular Bernoulli action $\Gamma \actson (X,\mu)$ of a biexact group, the crossed product state $\vphi_\mu$ on $M = L^\infty(X,\mu) \rtimes \Gamma$ is a solid state. Having proven Theorem \ref{thm.main}, it is equally tempting to believe that one may recover the measure class $\class(\vphi_\mu)$ as an isomorphism invariant for any nonsingular Bernoulli crossed product $L^\infty(X,\mu) \rtimes \Gamma$ whenever $\mu$ is $\Lambda$-invariant for a nonamenable subgroup $\Lambda \subset \Gamma$. The following example shows that both statements are wrong. The example is very similar to the construction in \eqref{eq.my-bernoulli}, except that we consider the free product of two nonamenable groups.

\begin{example}\label{ex.not-solid}
Let $\Gamma = \Gamma_1 \ast \Gamma_2$ be an arbitrary free product of two countable nonamenable groups. The following construction provides a nonsingular Bernoulli action $\Gamma \actson (X,\mu)$ with the following properties.
\begin{enumlist}
\item The crossed product state $\vphi_\mu$ on $M = L^\infty(X,\mu) \rtimes \Gamma$ is not a solid state.
\item The measure $\mu$ is $\Gamma_1$-invariant. There exists an equivalent product measure $\mu' \sim \mu$ that is $\Gamma_2$-invariant. The measure classes $\class(\vphi_\mu)$ and $\class(\vphi_{\mu'})$ are not equivalent.
\end{enumlist}
Since $\Gamma_2$ is nonamenable, not every element of $\Gamma_2$ has order $2$. Fix an element $a \in \Gamma_2$ of order at least $3$ (and possibly of infinite order). Define the map $\pi : \Gamma \to \Gamma_2$ by $\pi(h) = h$ for all $h \in \Gamma_2$ and $\pi(w h) = h$ whenever $h \in \Gamma_2$ and $w \in \Gamma$ is a reduced word in the free product $\Gamma_1 \ast \Gamma_2$ ending with a letter from $\Gamma_1 \setminus \{e\}$. Let $Y$ be a standard Borel space with equivalent probability measures $\nu \sim \eta$ on $Y$. Assume that $\nu$ and $\eta$ are not concentrated on a single atom. We will specify these measures later.

Define the subset $W \subset \Gamma$ by $W = \pi^{-1}(\{e,a\})$. For every $g \in \Gamma$, define $\mu_g = \nu$ if $g \in W$ and $\mu_g = \eta$ if $g \in \Gamma \setminus W$. Since $\pi(g v) = \pi(v)$ for all $g \in \Gamma_1$ and $v \in \Gamma$, we have that $g W = W$ for all $g \in \Gamma_1$. When $h \in \Gamma_2$, one has $h W \setminus W = \{h , ha\} \setminus \{e,a\}$ and $W \setminus h W = \{e,a\} \setminus \{h,ha\}$. We conclude that $g W \sdif W$ is a finite set for every $g \in \Gamma$. So, $\Gamma \actson (X,\mu) = \prod_{g \in \Gamma} (Y,\mu_g)$ is a nonsingular Bernoulli action. The action is essentially free. By construction, the measure $\mu$ is $\Gamma_1$-invariant. So, $\Gamma_1 \actson (X,\mu)$ is a pmp Bernoulli action, which is thus ergodic. A fortiori, $\Gamma \actson (X,\mu)$ is ergodic.

Next define $W' = W \setminus \Gamma_2 = W \setminus \{e,a\}$. Define $\mu'_g = \nu$ if $g \in W'$ and $\mu'_g = \eta$ if $g \in \Gamma \setminus W'$. Define the product measure $\mu' = \prod_{g \in \Gamma} \mu'_g$. Since $W' \sdif W$ is a finite set, we have that $\mu' \sim \mu$. We now have by construction that $h W' = W'$ for all $h \in \Gamma_2$. So, the measure $\mu'$ is $\Gamma_2$-invariant. Also note that for every $g \in \Gamma_1 \setminus \{e\}$, we have that $g W' \setminus W' = \{e,a\}$ and $W' \setminus g W' = \{g,ga\}$.

Since $\Gamma_1 \actson (X,\mu)$ and $\Gamma_2 \actson (X,\mu')$ are ergodic, the centralizer of both states $\vphi_\mu$ and $\vphi_{\mu'}$ is a nonamenable factor. We prove that for the appropriate choice of $\nu$ and $\eta$, we have $\class(\vphi_\mu) \not\prec \class(\vphi_{\mu'})$. It then follows from point 2 of Proposition \ref{prop.class} that $\vphi_\mu \not\prec \vphi_{\mu'}$. Since the relation $\prec$ between states is symmetric, also $\vphi_{\mu'} \not\prec \vphi_\mu$, so that $\vphi_\mu$ is not a solid state.

The measure classes $\class(\vphi_\mu)$ and $\class(\vphi_{\mu'})$ can be easily computed as follows. Define $\gamma = (\log d \nu / d \eta)_*(\eta)$. For every $g \in \Gamma$, we have that
\begin{alignat*}{2}
& (\log d(g \cdot \mu) / d \mu)_*(\mu) \sim \gamma^{\ast k} \ast \gammatil^{\ast l}  \quad &&\text{with $k = |g W \setminus W|$ and $l = |W \setminus g W|$,}\\
& (\log d(g \cdot \mu') / d \mu')_*(\mu') \sim \gamma^{\ast k} \ast \gammatil^{\ast l}  \quad &&\text{with $k = |g W' \setminus W'|$ and $l = |W' \setminus g W'|$.}
\end{alignat*}
A direct computation shows that $|g W \setminus W| = |W \setminus gW|$ for all $g \in \Gamma$ and that all elements of $\{0,1,2,\ldots\}$ appear as values. On the other hand, $|g W' \setminus W'| = |W' \setminus gW'|$ for all $g \in \Gamma$ but only the elements of $\{0,2,3,\ldots\}$ appear as values. We conclude that
$$\class(\vphi_\mu) = \delta_0 \vee \bigvee_{k=1}^\infty (\gamma \ast \gammatil)^{\ast k} \quad\text{and}\quad \class(\vphi_{\mu'}) = \delta_0 \vee \bigvee_{k=2}^\infty (\gamma \ast \gammatil)^{\ast k} \; .$$
Assume that $\gamma$ is nonatomic and that $K \subset \R$ is an independent Borel set such that $\gamma(K) > 0$. Denote by $\gamma_0$ the restriction of $\gamma$ to $K$. Clearly, $\gamma_0 \ast \widetilde{\gamma_0} \prec \class(\vphi_\mu)$. We claim that $\gamma_0 \ast \widetilde{\gamma_0}$ is orthogonal to $\class(\vphi_{\mu'})$. To prove this claim, it suffices to observe that for every $x \in \R \setminus \{0\}$, the set $(x + (K-K)) \cap (K-K)$ is contained in finitely many translates of $K \cup (-K)$. Arguing as in the proof of Lemma \ref{lem.indep-K}, it follows that $(\eta \ast \gamma \ast \gammatil)(K-K) = 0$ for every nonatomic probability measure $\eta$. So, the restriction of $\class(\vphi_{\mu'})$ to $K-K$ equals $\delta_0$. On the other hand, $\gamma_0 \ast \widetilde{\gamma_0}$ is a nonatomic probability measure that is concentrated on $K-K$, hence proving the claim.

Using the construction around \eqref{eq.setP}, we can give concrete examples where $\gamma$ is a nonatomic probability measure that is supported on $K$.
\end{example}

\section{Conjugacy results and proof of Proposition \ref{prop.some-isomorphism}}

In this section, we prove Proposition \ref{prop.some-isomorphism}. We use the following well known lemma and provide a proof for completeness.

\begin{lemma}\label{lem.dissipative-Z}
Let $\eta \sim \nu$ be equivalent, but distinct probability measures on the standard Borel space $Y$. Define $\mu_n = \nu$ when $n \in \N$ and $\mu_n = \eta$ when $n \in \Z \setminus \N$. Then, the nonsingular Bernoulli action
$$\Z \actson (Z,\zeta) = \prod_{n \in \Z} (Y,\mu_n)$$
is totally dissipative.
\end{lemma}
\begin{proof}
Since $\eta \neq \nu$, we can choose a Borel set $U \subset Y$ such that $\eta(U) \neq \nu(U)$. Since $\eta \sim \nu$, we have that $\eta(U)$ and $\nu(U)$ are different from $0$ and $1$. Define the probability measures $\gamma_n$ on $\{0,1\}$ by $\gamma_n(0) = \nu(U)$ for $n \in \N$ and $\gamma_n(0) = \eta(U)$ for $n \in \Z \setminus \N$. Define the factor map $\pi : Y \to \{0,1\}$ by $\pi(y) = 0$ iff $y \in U$. By construction, $\pi_*(\mu_n) = \gamma_n$ for all $n \in \Z$. So, the nonsingular Bernoulli action $\Z \actson \prod_{n \in \Z} (\{0,1\},\gamma_n)$ is a factor of $\Z \actson (Z,\zeta)$. By \cite[Theorem 1]{Ham81}, the former is totally dissipative, so that also $\Z \actson (Z,\zeta)$ is totally dissipative.
\end{proof}

\begin{proof}[{Proof of Proposition \ref{prop.some-isomorphism}}]
Write $\mu_{i,n} = \nu_i$ when $n \in \N$ and $\mu_{i,n} = \eta_i$ when $n \in \Z \setminus \N$. Consider the nonsingular Bernoulli actions $\Z \actson^{\be_i} (Z_i,\zeta_i) = \prod_{n \in \Z} (Y_i,\mu_{i,n})$. Since $\Gamma \actson^{\al_i} (X_i,\mu_i)$ is isomorphic with the action associated in \eqref{eq.more-general-nonsingular} to $\Z \actson (Z_i,\zeta_i)$, it suffices to prove that there exists a measure preserving conjugacy between $\be_1$ and $\be_2$.

Denote by $\nu$ the probability measure on $\R$ given by $(\log d\nu_1 / d\eta_1)_*(\nu_1) = (\log d\nu_2 / d\eta_2)_*(\nu_2)$. Since
$$\int_\R \exp(-t) \, d\nu(t) = \int_{Y_i} \frac{d\eta_i}{d\nu_i} \, d\nu_i = 1 \; ,$$
we can define the equivalent probability measure $\eta \sim \nu$ on $\R$ such that $(d\eta / d\nu)(t) = \exp(-t)$. Then $\pi_i = \log d\nu_i / d\eta_i$ is a factor map $\pi_i : Y_i \to \R$ satisfying $(\pi_i)_*(\nu_i) = \nu$ and $(\pi_i)_*(\eta_i) = \eta$ for all $i \in \{1,2\}$.

Define the probability measures $\mu_n$ on $\R$ by $\mu_n = \nu$ when $n \in \N$ and $\mu_n = \eta$ when $n \in \Z \setminus \N$. Consider the nonsingular Bernoulli action $\Z \actson (Z,\zeta) = \prod_{n \in \Z} (\R,\mu_n)$. Then $\psi_i : (Z_i,\zeta_i) \to (Z,\zeta) : (\psi_i(z))_n = \pi_i(z_n)$ is a measure preserving, $\Z$-equivariant factor map.

Denote by $(\nu_{i,t})_{t \in \R}$ the disintegration of $\nu_i$ along the measure preserving factor map $\pi_i : (Y_i,\nu_i) \to (\R,\nu)$. Similarly, denote by $(\eta_{i,t})_{t \in \R}$ the disintegration of $\eta_i$. Then, the disintegration $(\zeta_{i,z})_{z \in Z}$ of $\zeta_i$ along the factor map $\psi_i$ is given by
$$\zeta_{i,z} = \prod_{n \in \N} \nu_{i,z_n} \times \prod_{n \in \Z \setminus \N} \eta_{i,z_n} \; .$$
We assumed that the function $\pi_i$ is not essentially one-to-one. We thus find $\eps > 0$ such that the set
$$U_i = \{z \in \R \mid \;\text{the largest atom of $\nu_{i,z}$ has weight less than $1-\eps$}\;\}$$
has positive measure, $\nu(U_i) > 0$. For $\zeta$-a.e.\ $z \in Z$, there are infinitely many $n \in \N$ with $z_n \in U_i$. It follows that for $\zeta$-a.e.\ $z \in Z$, the product measure $\zeta_{i,z}$ is nonatomic.

Denote by $\lambda$ the Lebesgue measure on $[0,1]$. By the classification of factor maps, at least going back to \cite{Mah83} (see also \cite[Theorem 2.2]{GM87}), we can choose a measure preserving isomorphism $\theta_i : (Z_i,\zeta_i) \to (Z \times [0,1],\zeta \times \lambda)$ such that $p_Z(\theta_i(z)) = \psi_i(z)$ for $\zeta_i$-a.e.\ $z \in Z_i$, where $p_Z(z,t) = z$ for all $(z,t) \in Z \times [0,1]$. We denote $\theta = \theta_2^{-1} \circ \theta_1$ and have found a measure preserving isomorphism $\theta : (Z_1,\zeta_1) \to (Z_2,\zeta_2)$ satisfying $\psi_2(\theta(z)) = \psi_1(z)$ for $\zeta_1$-a.e.\ $z \in Z_1$.

By Lemma \ref{lem.dissipative-Z}, the Bernoulli action $\Z \actson (Z,\zeta)$ is totally dissipative. We can thus choose a Borel set $U \subset Z$ such that the sets $(n \cdot U)_{n \in \Z}$ form a partition of $Z$, up to measure zero. Write $U_i = \psi_i^{-1}(U)$. Then, $U_i \subset Z_i$ is a fundamental domain for the action $\Z \actson^{\be_i} (Z_i,\zeta_i)$. By construction, $\theta(U_1) = U_2$, up to measure zero. We can thus, essentially uniquely, define the nonsingular isomorphism
$$\Theta : Z_1 \to Z_2 : \Theta(n \cdot z) = n \cdot \theta(z) \quad\text{if $z \in U_1$ and $n \in \Z$.}$$
By construction, $\Theta$ is $\Z$-equivariant. We claim that $\Theta$ is measure preserving.

By construction,
$$\frac{d(n \cdot \zeta_i)}{d\zeta_i} = \frac{d(n \cdot \zeta)}{d \zeta} \circ \psi_i \; .$$
Therefore,
$$\frac{d(n \cdot \zeta_2)}{d\zeta_2} \circ \theta = \frac{d(n \cdot \zeta)}{d\zeta} \circ \psi_2 \circ \theta = \frac{d(n \cdot \zeta)}{d\zeta} \circ \psi_1 = \frac{d(n \cdot \zeta_1)}{d\zeta_1} \; .$$
Since $\theta$ is measure preserving, it then follows that the maps $z \mapsto n \cdot \theta( (-n) \cdot z)$ are measure preserving for all $n \in \Z$. Hence, $\Theta$ is measure preserving and the claim is proven. This concludes the proof of the proposition.
\end{proof}

\begin{example}\label{ex.some-isomorphism}
Let $\gamma \in (0,1)$. On the finite set $Y_1 = \{1,2,3\}$, we consider the probability measures $\nu_1(1)=1/2$, $\nu_1(2) = \nu_1(3) = 1/4$ and $\eta_1(1) = \gamma$, $\eta_1(2)=\eta_1(3)=(1-\gamma)/2$.

On the interval $Y_2 = [0,1]$, we consider the probability measures $\nu_2 \sim \eta_2$ where $\nu_2$ is the Lebesgue measure and
$$\frac{d\eta_2}{d\nu_2}(t) = \begin{cases} 2 \gamma &\;\;\text{if $0 \leq t \leq 1/2$,}\\ 2(1-\gamma) &\;\;\text{if $1/2<  t \leq 1$.}\end{cases}$$
For every countable group $\Lambda$, the associated nonsingular Bernoulli actions $\Gamma \actson (Y_i^\Gamma,\mu_i)$ of $\Gamma = \Z \ast \Lambda$ admit a measure preserving conjugacy, even though one base space is finite and the other base space is diffuse.
\end{example}

\end{document}